\documentclass[a4paper,11pt]{article}
\usepackage{graphicx}
\usepackage{amsmath,amsthm,amssymb,enumerate}%, esint}
\usepackage{euscript,mathrsfs}
\usepackage{dsfont}
\usepackage{xcolor}
\usepackage{empheq}
\usepackage{fancyhdr}
\usepackage{geometry}
\usepackage{tcolorbox}

\usepackage{color}
 \catcode`\@=11 \@addtoreset{equation}{section}
 
 \catcode`\@=12

\usepackage[colorlinks=true, linkcolor=blue, citecolor=red, urlcolor=blue]{hyperref}

\usepackage{stmaryrd}
\allowdisplaybreaks

\usepackage[labelformat=simple]{subcaption}

\usepackage{enumitem}
\usepackage{url}

\usepackage{tikz}
\usetikzlibrary{patterns}
\usetikzlibrary{matrix,arrows,decorations.pathmorphing,positioning,arrows.meta,quotes}

\usepackage[normalem]{ulem}
\usepackage{cancel}

\newtheorem{Theorem}{Theorem}[section]
\newtheorem{Proposition}[Theorem]{Proposition}
\newtheorem{Lemma}[Theorem]{Lemma}
\newtheorem{Corollary}[Theorem]{Corollary}

\theoremstyle{definition}
\newtheorem{Definition}[Theorem]{Definition}

\newtheorem{Remark}[Theorem]{Remark}

\newcommand{\bTheorem}[1]{
\begin{Theorem} \label{T#1} }
\newcommand{\eT}{\end{Theorem}}

\newcommand{\bProposition}[1]{
\begin{Proposition} \label{P#1}}
\newcommand{\eP}{\end{Proposition}}

\newcommand{\bLemma}[1]{
\begin{Lemma} \label{L#1} }
\newcommand{\eL}{\end{Lemma}}
\newcommand{\bCorollary}[1]{
\begin{Corollary} \label{C#1} }
\newcommand{\eC}{\end{Corollary}}

\newcommand{\up}{{\rm up}}
\newcommand{\Up}{{\rm Up}}
\newcommand{\Fup}{{\rm F}_h^{\eps}}

\newcommand{\vrh}{\vr_h}
\newcommand{\vuh}{\vu_h}

\newcommand{\vB}{\mathbf{B}}

\newcommand{\vC}{\mathbf{C}}
\newcommand{\vE}{\mathbf{E}}
\newcommand{\vj}{\mathbf{j}}
\newcommand{\vBh}{\vB_h}
\newcommand{\vCh}{\vC_h}
\newcommand{\vEh}{\vE_h}
\newcommand{\vjh}{\vj_h}

\newcommand{\bRemark}[1]{
\begin{Remark} \label{R#1} }
\newcommand{\eR}{\end{Remark}}

\newcommand{\bDefinition}[1]{\begin{Definition} \label{D#1} }
\newcommand{\eD}{\end{Definition}}

\newcommand{\av}{\avc}
\newcommand{\avg}[1]{ \left\{\hspace{-0.3em}\left\{ #1 \right\}\hspace{-0.3em}\right\}_\sigma }
\newcommand{\avc}[1]{ \widehat{ #1 } }

\newcommand{\co}[2]{{\rm co}\{ #1 , #2 \}}
\newcommand{\Ov}[1]{\overline{#1}}

\newcommand{\auh}{ \widehat{\vuh} }
\newcommand{\avh}{ \widehat{\vh} }

\newcommand{\us}{ u_\sigma }

\newcommand{\intSh}[1] {\int_{\sigma} #1 \ds }

\newcommand{\bu}{\mathbf u}

\newcommand{\bfu}{\mathbf{u}}
\newcommand{\bfv}{\mathbf{v}}

\newcommand{\bfn}{\mathbf{n}}
\newcommand{\nG}{\bfn_{\sigma}}
\newcommand{\vv}{\bfv}

\newcommand{\vh}{\vv_h}

\newcommand{\bfx}{\mathbf{x}}

\newcommand{\xx}{\bfx}

\newcommand {\CR} {V_h}
\newcommand {\ND} {\mathcal{N}_h}
\newcommand {\RT} {\mathcal{R}_h}
\newcommand {\FS} {X_h}

\newcommand {\CRz} {V_{0,h}}
\newcommand {\NDz} {\mathcal{N}_{0,h}}
\newcommand {\RTz} {\mathcal{R}_{0,h}}
\newcommand {\FSz} {X_{0,h}}

\newcommand{\bfphi}{\boldsymbol{\phi}}

\newcommand{\bfpsi}{\boldsymbol{\psi}}

\newcommand{\PiQ}{\Pi_Q}
\newcommand{\PiV}{\Pi_V}
\newcommand{\PiN}{\Pi_{\cal N} }
\newcommand{\PiR}{\Pi_{\cal R} }
\newcommand{\PiVv}{\PiV \vv}

\newcommand{\Hdiv}{  H({\rm curl}) }
\newcommand{\Hcurl}{ H({\rm div}) }

\newcommand{\HdivO}{  H({\rm curl};\Omega) }
\newcommand{\HcurlO}{ H({\rm div};\Omega) }

\newcommand{\ds}{\,{\rm d}S(x)}

\newcommand{\bFormula}[1]{
\begin{equation} \label{#1}}
\newcommand{\eF}{\end{equation}}

\newcommand{\grid}{\Oh}
\newcommand{\Oh}{\Omega_h}

\newcommand{\TS}{\Delta t}

\newcommand{\Divh}{{\rm div}_h}
\newcommand{\Curlh}{{\bf curl}_h}
\newcommand{\Gradh}{\nabla_h}

\newcommand{\aleq}{\stackrel{<}{\sim}}

\newcommand{\vr}{\varrho}

\newcommand{\vu}{\vc{u}}
\newcommand{\vm}{\vc{m}}
\newcommand{\vn}{\vc{n}}
\newcommand{\vc}[1]{{\mathbf #1}}
\newcommand{\Div}{{\rm div}}
\newcommand{\Grad}{\nabla}
\newcommand{\Curl}{{\bf curl}}

\newcommand{\dx}{\,{\rm d} {x}}

\newcommand{\dt}{\,{\rm d} t }

\newcommand{\jump}[1]{\left\llbracket#1\right\rrbracket}
\newcommand{\abs}[1]{\left\lvert#1\right\rvert}

\newcommand{\norm}[1]{\left\lVert#1\right\rVert}

\newcommand{\dxdt}{\dx \dt}
\newcommand{\intO}[1]{\int_{\Omega} #1 \dx }

\newcommand{\Om}{\Omega}

\newcommand{\intK}[1]{\int_{K} #1 \ \dx}
\newcommand{\intG}[1]{\int_{\sigma} #1 \ \ds}
\newcommand{\intOB}[1]{\int_{\Omega} \left( #1 \right) \ \dx}
\newcommand{\intTO}[1]{\int_0^T \int_{\Omega} #1 \ \dxdt}
\newcommand{\intTOB}[1]{ \int_0^T \int_{\Omega} \left( #1 \right) \ \dxdt}

\newcommand{\R}{\mathbb{R}}

\renewcommand{\S}{\mathbb{S}}

\newcommand{\newcom}{\newcommand}
\newcommand{\beq}{\begin{equation}}
\newcommand{\eeq}{\end{equation}}
\newcom{\ben}{\begin{eqnarray}}
\newcom{\een}{\end{eqnarray}}
\newcom{\beno}{\begin{eqnarray*}}
\newcom{\eeno}{\end{eqnarray*}}
\newcom{\bali}{\begin{aligned}}
\newcom{\eali}{\end{aligned}}

\newcommand{\f}{\frac}

\definecolor{Cgrey}{rgb}{0.85,0.85,0.85}
\definecolor{Cblue}{rgb}{0.50,0.85,0.85}
\definecolor{Cred}{rgb}{1,0,0}
\definecolor{fancy}{rgb}{0.10,0.85,0.10}
\definecolor{forestgreen}{rgb}{0.13, 0.55, 0.13}

\newcommand{\cblue}{\color{blue}}

\date{}
\newcommand{\pd}{\partial}
\newcommand{\pdt}{\pd _t}

\newcommand{\Hc}{\mathcal{H}}
\newcommand{\eps}{\varepsilon}
\newcommand{\faces}{\mathcal{E}}

\newcommand{\facesK}{\faces(K)}

\newcommand{\facesint}{\faces^I}

\newcommand{\facesext}{\faces^B}

\newcommand{\mcE} {\faces}
\newcommand{\mcEe} {\facesext}
\newcommand{\mcEi} {\facesint}
\newcommand{\mcP} {\mathcal{P}}

\begin{document}

%%%%%%%%%%%%%%%%%%%%%%%%%%%%%%%%

\pagestyle{fancy} \lhead{\color{blue}{Convergent numerical solutions to compressible MHD system}}
\rhead{\emph{Y.Li and B.She}}

\title{
\bf
On convergence of numerical solutions for the compressible MHD system with exactly divergence-free magnetic field
}

\author{
Yang Li$^1$ \quad \,\,\,\,\,\,\,\,\,\,\,\,\,\,\,\,\,\,\,\,\,\,\,\,  Bangwei She$^{2,3}$ \\ \\ $^1$School of Mathematical Sciences, \\ Anhui University,  230601, Hefei, People's Republic of China \\ Email: lynjum@163.com\\ \\
$^{2}$Institute of Mathematics, \\ Czech Academy of Sciences, \v Zitn\'a 25, 115 67, Praha 1, Czech Republic \\ and \\
$^3$Department of Mathematical Analysis,  Charles University\\ Sokolovsk\'{a} 83, 186 75, Praha 8, Czech Republic \\
 Email: she@math.cas.cz \\ \\
}

\date{\today}

\maketitle

\begin{abstract}
We study a general convergence theory for the  numerical solutions of  compressible viscous and electrically conducting fluids with a focus on  numerical schemes that preserve the divergence free property of magnetic field exactly. 
Our strategy utilizes the recent concepts of dissipative weak solutions and consistent approximations. 
First, we  show the dissipative weak--strong uniqueness principle, meaning a dissipative weak solution coincides with a classical solution as long as they emanate from the same initial data.  
Next, we show the convergence of consistent approximation towards the dissipative weak solution and thus the classical solution. 
Upon interpreting the consistent approximation as the stability and consistency of suitable numerical solutions we have established a generalized Lax equivalence theory: convergence $\Longleftrightarrow$ stability and consistency. 
Further, to illustrate the application of this theory, we propose two novel mixed finite volume-finite element methods with exact divergence-free magnetic field. 
Finally, by showing solutions of these two schemes are consistent approximations, we conclude their convergence towards the dissipative weak solution and the classical solution.

\end{abstract}


\tableofcontents

\section{Introduction}
Motivated by its wide applications in astrophysics and plasma physics, we study the numerical theory for the magnetohydrodynamics (MHD) system describing the motion of compressible viscous and electrically conducting fluids. 
Let $t\in (0,T)$ and $x\in \Omega \subset \R^d\,(d=2,3)$  be respectively the time and space variables. We denote by $\vr=\vr(t,x)$ the density of the fluids, $\vu=\vu(t,x)\in \R^d$ the velocity field, $\vB=\vB(t,x)\in \R^d$ the magnetic field,  $\mathbf{E}=\mathbf{E}(t,x)\in \R^d$ the electric field and $\mathbf{j}=\mathbf{j}(t,x)\in \R^d$ the current density. 
The compressible MHD system admits the following B-E form
 \begin{equation}\label{pde}
\left\{\begin{aligned}
& \pdt \vr+\Div (\vr \vu)=0,\\
& \pdt (\vr\vu)+\Div (\vr \vu \otimes \vu )+\Grad p(\vr)=\Div \mathbb{S}(\Grad \vu)+\mathbf{j}\times \vB,\\
& \mathbf{j}=\Curl \vB, \\
& \pdt \vB+\Curl \mathbf{E}=\mathbf{0},\\
& \Div \vB=0,\\
\end{aligned}\right.
\end{equation}
in the time-space cylinder $Q_T=(0,T)\times \Omega$ with  $\vj = \vE + \vu \times \vB$.  
Here, $\mathbb{S}=\mathbb{S}(\Grad \vu)$ stands for the Newtonian viscous stress tensor
\[
\mathbb{S}(\Grad \vu)=\mu \left( \Grad \vu+\Grad^T\vu-\f{2}{d}\Div \vu \mathbb{I}  \right)+\lambda \Div \vu \mathbb{I}, \; \mu >0, \frac{d-2}{d} \mu + \lambda \geq0. 
\]
%where $\mu>0$ and $\lambda \geq 0$ are the shear and bulk viscosity coefficients, respectively.
Moreover, the pressure $p=p(\vr)$ is assumed to satisfy the isentropic law
\begin{equation}\label{plaw}
    p(\vr)= a \vr^\gamma, \quad a>0,
\end{equation}
where $\gamma>1$ is the adiabatic exponent.
System \eqref{pde} is supplemented with the boundary conditions ($\mathbf{n}$ denotes the unit outward normal on the boundary $\pd \Omega$):
%\begin{equation}\label{BCs}
% \vu|_{\pd \Omega}=\mathbf{0},\quad
% \vB\cdot \mathbf{n}|_{\pd \Omega}=0,\,\,
% \mathbf{E}\times \mathbf{n}|_{\pd \Omega}=\mathbf{0},\\
%\end{equation}
\begin{equation}\label{BCs}
\begin{cases}
\mbox{ either periodic boundary conditions with  } & \Omega = {\cal T} ^d = \left( [0,1]_{\{0,1\}} \right)^d,
\\
\mbox{ or non-slip boundary conditions }  &  \vu|_{\pd \Omega}=\mathbf{0},\,\,
 \vB\cdot \mathbf{n}|_{\pd \Omega}=0,\,\, 
 \mathbf{E}\times \mathbf{n}|_{\pd \Omega}=\mathbf{0},
\end{cases}
\end{equation}
together with the initial data
\begin{equation}\label{ini_c}
(\vr,\vr\vu,\vB)|_{t=0}=(\vr_0,\mathbf{m}_0,\vB_0).
%(\vr,\vr\vu,\vB,\vE)|_{t=0}=(\vr_0,\mathbf{m}_0,\vB_0, \Curl \vB_0 - \vu_0 \times \vB_0).
\end{equation}

In contrast to the rich numerical convergence theory of  incompressible MHD problems, see e.g.~\cite{BCP, HQS, Prohl, ZYB}, the numerical analysis of  compressible MHD system~\eqref{pde} is open in general. The only result in literature is the convergence of a finite element approximation towards (a suitable subsequence of) weak solutions recently  reported by Ding and Mao~\cite{Ding} based on the original form\footnote{By ``original form" we mean that the third and fourth equations of the system \eqref{pde} are replaced by $\pdt \vB + \Curl (\Curl \vB - \vu \times \vB)=0$.} of the MHD system. 
We point out that the study in \cite{Ding} requires the technical assumption on the adiabatic exponent $\gamma>3$ that unfortunately excludes the physical gas parameter, e.g. $\gamma=7/5$ for diatomic gas and $\gamma=5/3$ for monatomic gas. 

The aim of this paper is to build a general convergence theory for the MHD system \eqref{pde} in spirit of the celebrated Lax equivalence theorem. 
Our strategy leans on the recent concepts of consistent approximation, dissipative weak solution, and the so-called dissipative weak--strong uniqueness principle developed by Feireisl et al.~\cite{FeLMMiSh} in the context of compressible Navier--Stokes equations.  

The current paper is a continuation of our previous work \cite{LiShe_MHD1}, where we studied a general convergence theory for numerical solutions that preserve the divergence free of magnetic field weakly. Here, we are interested in designing numerical schemes that preserve the divergence free of magnetic field exactly and providing a general convergence theory for such numerical solutions. 

The highlights of the paper reads:  
\begin{itemize}[noitemsep,topsep=0pt]
\item[i)] We introduce the concept of dissipative weak solution to the MHD system and prove that a dissipative weak solution coincides with the classical solution of the same problem, see Theorem~\ref{Th1};
\item[ii)] We establish a generalized Lax equivalence theory for a large class of numerical solutions, see Theorem~\ref{Th2};
\item[iii)] We propose two invariant domain preserving schemes, see {\bf Scheme-I} and {\bf Scheme-II} given later in Section \ref{sec_con}, meaning that the numerical schemes preserve the positivity of density,  conservation of the mass, stability of the total energy, and divergence free of the magnetic field; 
\item[iv)] We apply Theorem~\ref{Th2} to the convergence analysis of numerical solutions of two mixed finite volume -- finite element methods, that also indicates the global-in-time existence of  dissipative weak solutions for the full range of physically relevant adiabatic exponent $\gamma \in (1,\infty)$. 

%\item[ii)] The theory suits for the whole range of physically relevant adiabatic exponent $\gamma \in (1, \infty)$, which also indicates the global-in-time existence of dissipative weak solutions for all $\gamma>1$. 
%\item[ii)] The theory holds for a large class of numerical solutions that falls into the class of consistent approximations given below in Definition~\ref{def_ca}.  
%\item[iii)] Our numerical schemes preserve on the discrete level the energy stability, exact divergence free of magnetic field, conservation of total mass, and positivity of density.
\end{itemize}

The rest of this paper is organized as follows. 
In Section \ref{sec_mr} we introduce the definition of dissipative weak solutions and consistent approximation to the MHD system \eqref{pde}, followed by the main results. 
In Section \ref{sec_ws}, we establish the relative energy inequality in the class of dissipative weak solutions and prove the weak--strong uniqueness principle. 
In Section \ref{sec_con} we prove the convergence of a consistent approximation, and show its application to the convergence analysis of numerical solutions in terms of two mixed finite volume -- finite element  methods. Section \ref{sec_end} is the conclusion of the paper.

\section{Main results}\label{sec_mr}
In this section, we introduce the concepts of dissipative weak solution and consistent approximation for the compressible MHD system followed by the main results. 
\subsection{Dissipative weak solution and consistent approximation}
Let $\mathcal{M}\left( \overline{\Om}\right)$ signifies the space of signed Borel measures over $ \overline{\Om}$ and $\mathcal{M}^{+}\left( \overline{\Om}\right)$ means the non-negative ones. We now introduce the concept of \emph{dissipative weak solutions}.
\begin{Definition}[Dissipative weak solution]
A quadruple $(\vr,\vu,\vB,\mathbf{E})$ is said to be a dissipative weak solution to the compressible MHD system \eqref{pde}--\eqref{ini_c} if:
\begin{itemize}
\item{ Regularity class}
\[
\vr \geq 0,\,\, \vr \in L^{\infty}(0,T;L^{\gamma}(\Om)),\,\,\sqrt{\vr}\vu \in L^{\infty}(0,T;L^{2}(\Om;\R^d)),
\]
\[
\vB \in L^{\infty}(0,T;L^{2}(\Om;\R^d)),\,\,\Grad \vu \in L^2(0,T;L^{2}(\Om;\R^{d\times d})),
\]
\[
\mathbf{j}  \in L^2(0,T;L^{2}(\Om;\R^{d})),\,\,\,\mathbf{E}\in  L^2(0,T;L^{\f{3}{2}  }(\Om;\R^{d}));
\]
\item{ The continuity equation}
\beq\label{lbb3}
\int_0^{\tau}\int_{\Om} \Big ( \vr \partial_t \varphi + \vr\vu\cdot \Grad \varphi  \Big )\dxdt
=\left[ \int_{\Om} \vr \varphi \dx\right]_{t=0}^{t=\tau}
\eeq
for a.e. $\tau \in (0,T)$, any $\varphi \in C^1([0,T]\times \overline{\Om})$;

\item { The momentum equation}
\[
\int_0^{\tau}\int_{\Om} \Big ( \vr\vu\cdot \partial_t \vv + (\vr\vu\otimes \vu): \Grad \vv+p(\vr)\Div \vv-\mathbb{S}(\Grad\vu):\Grad \vv
+\left(\mathbf{j} \times \mathbf{B}\right)\cdot \vv
 \Big ) \dxdt
\]
\beq\label{lbb4}
+
 \int_0^{\tau}\int_{\overline{\Om}} \Grad \vv :{\rm d} \mu_c(t)\dt
 +
 \int_0^{\tau}\int_{\overline{\Om}} \vv \cdot {\rm d} \mu_B(t)\dt
=\left[ \int_{\Om}  \vr\vu \cdot\vv   \dx \right]_{t=0}^{t=\tau}
\eeq
for a.e. $\tau \in (0,T)$ and for any $\vv \in C_c^1([0,T]\times \Om;\R^d)$,  where $\vj = \vE + \vu \times \vB$. Here 
 $\mu_c\in L^{\infty}(0,T;\mathcal{M}(\overline{\Om};\R^{d \times d}_{sym}))$ and $\mu_B\in L^{2}(0,T;\mathcal{M}(\overline{\Om};\R^{d \times d}))$ reflect the {\em concentration/oscillation} defects;

\item {  The Maxwell equation}
\beq\label{lbb5}
\int_0^{\tau}\int_{\Om} \Big ( \vB\cdot \pdt \vc{\varphi}-   \mathbf{E}\cdot \Curl  \vc{\varphi} \Big) \dxdt
=\left[ \int_{\Om}  \vB\cdot\vc{\varphi}   \dx \right]_{t=0}^{t=\tau}
\eeq
for a.e. $\tau \in (0,T)$, any $\vc{\varphi} \in C_c^1([0,T]\times \overline{\Om};\R^d),\vc{\varphi}\cdot \mathbf{n}|_{\partial \Omega}=0$;

\item {  Amp\`ere's law}
\beq\label{lbb5.1}
\int_0^{\tau}\int_{\Om}   \mathbf{j} \cdot \vc{\varphi} \dxdt
=\int_0^{\tau}\int_{\Om}   \vB\cdot
\Curl \vc{\varphi} \dxdt
\eeq
for a.e. $\tau \in (0,T)$, any $\vc{\varphi} \in C_c^1([0,T]\times \overline{\Om};\R^d),\vc{\varphi}\cdot \mathbf{n}|_{\partial \Omega}=0$;

\item { Divergence--free of magnetic field
\beq\label{lbb6}
\Div \vB(t) =0
\eeq
for a.e. $t\in (0,T)$;
}

\item { Balance of total energy 
\[
\int_{\Om} \left[\f{1}{2}\vr |\vu|^2+\f{1}{2}|\mathbf{B}|^2+H(\vr)\right](\tau,x)\dx
+\int_0^{\tau}\int_{\Om} \left (  \mathbb{S}(\Grad\vu):\Grad \vu
+ \left| \mathbf{j} \right|^2
\right )\dxdt
\]
\beq\label{lbb7}
+\int_{\overline{\Om} }{\rm d}\mathfrak{D}(\tau)
+ \int_0^{\tau}
  \int_{\overline{\Om}}  {\rm d}\mathfrak{C}
\leq
\int_{\Om}\left[ \f{1}{2}\f{|\mathbf{m}_0|^2}{\vr_0}+\f{1}{2}|\mathbf{B}_0|^2+H(\vr_0)\right]\dx
\eeq
for a.e. $\tau \in (0,T)$ and some {\em energy defects} $\mathfrak{D}\in L^{\infty}(0,T; \mathcal{M}^{+}(\overline{\Om}))$, $\mathfrak{C} \in \mathcal{M}^{+}([0,T]\times \overline{\Om})$; $H(\vr)$ is the potential energy given by
\[
H(\vr):=\f{a}{\gamma-1}\vr^{\gamma};
\]
}

\item { Compatibility conditions
%\beq\label{lbb8.1}
%\mathbf{j}=\mathbf{E}+\vu \times \vB \,\,\,\,\,\, \text{    a.e. in   } (0,T)\times \Om,
%\eeq
\beq\label{lbb8}
 \int_0^T \psi(t)\int_{\overline{\Om}}{\rm d}|\mu_c(t)|\dt \lesssim \int_0^T \psi(t)\int_{\overline{\Om}}{\rm d}\mathfrak{D}(t)\dt,
\eeq
\beq\label{lbb8-2}
 \int_0^T \psi(t)\int_{\overline{\Om}}{\rm d}|\mu_B(t)|\dt  \lesssim  \f{C}{\epsilon}\int_0^T \psi(t)\int_{\overline{\Om}}{\rm d}\mathfrak{D}(t)\dt +\epsilon  \int_0^T
  \int_{\overline{\Om}} \psi(t) {\rm d}\mathfrak{C}
\eeq
for any $\epsilon>0,\psi \in C([0,T]),\psi \geq 0$.
}
\end{itemize}
\end{Definition}

\begin{Remark}
We give some comments on the measures appearing above. In the momentum equation, we denote by $\mu_c$ the concentration/oscillation defects coming from the nonlinear terms $\vr \vu\otimes \vu,p(\vr)$; and $\mu_B$ reveals that of $\mathbf{j}\times \vB$. In the balance of total energy, the non-negative measure $\mathfrak{D}$ stands for the defects from the total energy $\f{1}{2}\vr |\vu|^2+\f{1}{2}|\mathbf{B}|^2+H(\vr)$, while $\mathfrak{C}$ includes the defects from $\mathbb{S}(\Grad\vu):\Grad \vu
+ \left| \mathbf{j} \right|^2$. Motivated by suitable numerical schemes, these measures are interrelated through the compatibility conditions. This is crucial for showing the weak-strong uniqueness property. 
\end{Remark}

%%%%%%%%%%%%%%%%%%%%%%
Next, we introduce the concept of consistent approximation for the compressible MHD system.
\begin{Definition}[Consistent approximation]\label{def_ca}
Let the discrete operator $\Gradh$ (resp. $\Divh$ and $\Curlh$) be compatible\footnote{A simple example of such compatibility is that $\Gradh = \Grad$ element-wisely.} with the continuous differential operator $\Grad$ (resp. $\Div$  and $\Curl$) in the sense of \cite[Definition 5.8]{FeLMMiSh}.
 We say a numerical approximation $(\vrh,\vuh,\vBh, \vEh)$ is a consistent approximation of the MHD system~\eqref{pde}--\eqref{ini_c} if the following stability and consistency conditions hold:
\begin{enumerate}
\item {\bf Stability.}
The numerical approximation is stable in the sense that
\begin{equation}\label{es}
E_h(\tau) +  \int_0^\tau \intOB{ \S(\Gradh \vuh) : \Gradh \vuh
+   |\vjh |^2 } \leq E_h(0),  \quad \forall \tau \in [0,T],
\end{equation}
where $ E_h =   \intO{ \left(\frac{1}{2} \vrh  \abs{\Pi_h \vuh }^2  + \mathcal{H}(\vrh )+\frac{1}{2} \abs{\vBh }^2 \right)  } $  the total energy,  $\Hc(\vrh) =  \frac{a}{\gamma-1}\vrh^\gamma$ the pressure potential, and  $\Pi_h$ is either identity or a piecewise constant projection operator satisfying $\norm{\Pi_h \vuh - \vuh }_{L^2(\Om)} \leq h \norm{\Gradh \vuh}_{L^2(\Om)}$.
\item {\bf Consistency.}
The numerical approximation is consistent if: 

\begin{subequations}\label{cP}
$\bullet$ {\bf Continuity equation.} It holds for  any $\phi \in C_c^1([0,T) \times \Ov{\Omega})$ that
\begin{equation} \label{cP1}
 \intTOB{  \vrh \partial_t \phi + \vrh \vuh \cdot \Grad \phi }
 =- \intO{ \vrh^0 \phi(0,\cdot) }
+  e_{1,h}[\phi],
\end{equation}
where $e_{1,h}[\phi] \to 0 \mbox{ as } h \to 0 \mbox{ for any } \phi \in C_c^M([0,T) \times \Ov{\Omega})  \mbox{ for some integer } M\geq 1;$

$\bullet$ {\bf Balance of momentum.}  It holds for any $\vv \in C_c^1([0,T) \times \Omega; \R^d)$ that
\begin{equation} \label{cP2}
\begin{aligned}
 \intTOB{  \vrh \Pi_h \vuh \cdot \partial_t \vv + \vrh \Pi_h \vuh \otimes \vuh  : \Grad \vv  + p_h \Div \vv
 -    \S( \Gradh \vuh) : \Grad \vv }
\\  +\intTO{  \vjh \times \vBh \cdot  \vv}
= - \intO{ \vrh^0 {\Pi_h \vuh^0} \cdot \vv(0,\cdot) }
+e_{2,h}[\vv]
\end{aligned}
\end{equation}
with $\vjh = \vEh + \vuh \times \vBh$,  
where $e_{2,h}[\vv] \to 0$ as $ h \to 0$  for any $ \vv \in C_c^M([0,T) \times \Omega; \R^d)$  for some integer  $M\geq 1$;

$\bullet$ {\bf The Maxwell equation.} It holds for any $\vC \in C_c^1([0,T) \times \Ov{\Omega}; \R^d)$,  $\vC \cdot \vn|_{\pd \Omega}=0$ that
 \begin{equation} \label{cP3}
   \intTOB{ \vBh \cdot \pdt \vC  - \vEh  \cdot \Curlh \vBh }
=-  \intO{ \vBh^0 \cdot \vC(0,\cdot)}  +e_{3,h}[\vC]
 \end{equation}
 where $e_{3,h}[\vC] \to 0$ as $ h \to 0$  for any $ \vC \in C_c^M([0,T) \times \overline{\Omega}; \R^d),\vC \cdot \vn|_{\pd \Omega}=0$  for some integer  $M\geq 1$;

$\bullet$ {\bf Amp\`ere's law.} It holds  for any $\bfpsi \in C_c^1([0,T) \times \Ov{\Omega}; \R^d)$,  $\bfpsi \cdot \vn|_{\pd \Omega}=0$ that 
\begin{equation} \label{cP4}
 \intO{\vjh \cdot \bfpsi - \vBh \cdot\Curl \bfpsi }
= e_{4,h}[\psi]
 \end{equation}
where $e_{4,h}[\psi] \to 0$ as $ h \to 0$  for any $ \psi \in C^M(\Ov{\Omega})\cap L^2_0(\Omega)$ for some integer  $M\geq 1$;

$\bullet$ {\bf  Divergence free of magnetic field.} It holds that
 \begin{equation}\label{cP5}
 \Divh \vBh =0 .
 \end{equation}
 \end{subequations}
\end{enumerate}
\end{Definition}
\medskip

%%%%%%%%%%%%%%%%%%%%%%

\subsection{Main results}
Now we are ready to present the main results. The first one states the weak--strong uniqueness property. 
\begin{Theorem}[Weak--strong uniqueness]\label{Th1}
Let $(\widetilde{\vr},\widetilde{\vu},\widetilde{\vB},
\widetilde{\mathbf{E}})$  be subject to 
\begin{equation}\label{STC}
\left\{\begin{aligned}
& \widetilde{\vr} \in C^1([0,T]\times \overline{\Om}),\,\,\widetilde{\vr}>0,\\
& \widetilde{\vu}\in C^1([0,T]\times \overline{\Om};\R^d),\,\,\widetilde{\vu} |_{\pd \Om}=\vc{0},\\
& \widetilde{\vB}\in C^1([0,T]\times \overline{\Om};\R^d),\,
\widetilde{\vB} \cdot \mathbf{n}|_{\partial \Omega}=0,\,\,\Div  \widetilde{\vB} =0.\\
\end{aligned}\right.
\end{equation}
be a classical solution to \eqref{pde}--\eqref{ini_c} starting from the smooth initial data $(\vr_0,\vu_0,\vB_0)$ with strictly positive $\vr_0$ and $\Div \vB_0=0$. Let $(\vr,\vu,\vB,\vE)$ be a dissipative weak solution to \eqref{pde}--\eqref{ini_c} emanating from $(\vr_0, \vr_0\vu_0,\vB_0)$. Then
\[
\vr=\widetilde{\vr},\,\,\,\,\vu=\widetilde{\vu},\,\,\,\,\vB=\widetilde{\vB},\,\,\,\,\vE=\widetilde{\vE}\,\,\,\,
\text{      in    } (0,T)\times \Om,
\]
\[
\mu_c=\mathbf{0},\,\,\mu_B=\mathbf{0},\,\, \mathfrak{D}=0,\,\, \mathfrak{C}=0.
\]
\end{Theorem}

 The second result is concerned with the convergence of a consistent approximation towards a DW solution as well as a classical solution.
\begin{Theorem}[Convergence of a consistent approximation]\label{Th2}
Let $(\vrh,\vuh,\vBh, \vEh)$ be a consistent approximation of the MHD system~\eqref{pde}--\eqref{ini_c} in the sense of Definition~\ref{def_ca}. Then the following convergence results hold:
\begin{enumerate}
\item {\bf Convergence to DW solution.} 
There exists a subsequence of  $(\vrh,\vuh,\vBh, \vEh)$, not relabelled, such that
\[
\begin{aligned}
\vrh \rightarrow & \; \vr \mbox{ weakly-(*) in } L^\infty(0,T;L^\gamma(\Om)),\quad
&& \vuh \rightarrow   \; \vu \mbox{ weakly in } L^2 ((0,T)\times \Om; \R^d), \\
\vBh \rightarrow &  \;  \vB \text{ weakly-(*) in } L^\infty(0,T; L^2( \Om; \R^d)),\quad
&& \vEh \rightarrow  \; \vE \text{ weakly in  } L^2 (0,T;L^{\f{3}{2}}( \Om;\R^d)). 
\end{aligned}
\]
\item {\bf Convergence to classical solution.}
In addition, let the MHD system \eqref{pde}--\eqref{ini_c} admit a classical solution in the class \eqref{STC}. 
Then the above weak limit is unconditional (no need of subsequence but the whole sequence) and the limit quantity $(\vr,\vu,\vB,\vE)$ coincides with the classical solution.

\end{enumerate}
\end{Theorem}

\section{Weak--strong uniqueness}\label{sec_ws}
In this section we prove the first main result, that is the weak--strong uniqueness principle stated in Theorem~\ref{Th1}. The main tool we use in the proof is the so-called \emph{relative entropy functional} developed by Feireisl et al.~\cite{FJN} in the context of compressible Navier--Stokes system. 
We shall introduce the relative entropy functional for compressible MHD system in the context of DW solutions. To fix ideas, we proceed in case of Dirichlet boundary conditions, while the periodic case can be carried out analogously. 

\subsection{Relative energy inequality}\label{rel_ener}
The goal of this subsection is to establish the relative energy inequality in the context of dissipative weak solutions. To this end, let $(\vr,\vu,\vB, \vE)$ be a dissipative weak solution to \eqref{pde}--\eqref{ini_c} emanating from $(\vr_0,\mathbf{m}_0,\vB_0)$ and $(r,\mathbf{U},\mathbf{b})$ belongs to the regularity class {\cblue \eqref{STC}}. 
Similar to \cite{FNS} in the context of finite energy weak solutions, we introduce the \emph{relative entropy} in the framework of dissipative weak solutions
%We measure their difference by the following \emph{relative entropy} functional 
\[
\mathcal{E}\Big( (\vr,\vu,\vB)\,\Big|\,(r,\mathbf{U},\mathbf{b})\Big)(\tau)
\]
\beq\label{lbn2}
=\int_{\Om}\left(  \f{1}{2}\vr |\vu-\mathbf{U}|^2+\f{1}{2} |\vB-\mathbf{b}|^2
+H(\vr)-H(r)-H'(r)(\vr-r)
\right)(\tau,\cdot) \dx.
\eeq
We further rewrite the relative entropy in an equivalent form as follows
\[
\mathcal{E}\Big((\vr,\vu,\mathbf{b})\,\Big|\,(r,\mathbf{U},\mathbf{b})\Big)(\tau)
=\int_{\Om} \left(  \f{1}{2}\vr |\vu|^2+\f{1}{2}|\vB|^2+H(\vr) \right) \dx
+\int_{\Om}\f{1}{2} \vr  |\mathbf{U}|^2\dx
\]
\beq\label{lbn3}
-\int_{\Om} \vr \vu \cdot \mathbf{U} \dx
-\int_{\Om} \vB\cdot \mathbf{b} \dx
-\int_{\Om} \vr  H'(r) \dx
+\int_{\Om} p(r) \dx+\f{1}{2}\int_{\Om}|\mathbf{b}|^2  \dx.
\eeq
The crucial observation is that the integrals on the right-hand side of (\ref{lbn3}) can be expressed through (\ref{lbb3})-(\ref{lbb7}) with suitable choices of test functions. To handle the density-dependent terms, testing the continuity equation (\ref{lbb3}) by $\f{1}{2}|\mathbf{U}|^2$ gives
\beq\label{lbn4}
\left[ \int_{\Om} \f{1}{2} \vr  |\mathbf{U}|^2  \dx \right]_{t=0}^{t=\tau}
=\int_0^{\tau}\int_{\Om} \Big ( \vr\mathbf{U}\cdot \pdt \mathbf{U}+ \vr\vu \cdot \Grad \mathbf{U}  \cdot \mathbf{U}\Big )\dxdt.
\eeq
Analogously, we use $H'(r)$ as a test function in (\ref{lbb3}) to find that
\beq\label{lbn5}
\left[ \int_{\Om}  \vr H'(r)   \dx \right]_{t=0}^{t=\tau}
=\int_0^{\tau}\int_{\Om} \Big ( \vr \pdt H'(r)+ \vr\vu \cdot  \Grad H'(r) \Big )\dxdt.
\eeq
Upon choosing $\mathbf{U}$ as a test function in the momentum equation (\ref{lbb4}),
\[
\left[ \int_{\Om}  \vr\vu \cdot \mathbf{U}  \dx \right]_{t=0}^{t=\tau}
=\int_0^{\tau}\int_{\Om} \Big ( \vr\vu\cdot \pdt \mathbf{U} + \vr\vu\otimes \vu: \Grad \mathbf{U}+p(\vr)\Div \mathbf{U}
\]
\beq\label{lbn6}
-\mathbb{S}(\Grad\vu):\Grad \mathbf{U}
+( \mathbf{j} \times \vB )\cdot \mathbf{U}
 \Big ) \dxdt
+
 \int_0^{\tau}\int_{\overline{\Om}} \Grad \mathbf{U} :{\rm d} \mu_c(t)\dt
 +
 \int_0^{\tau}\int_{\overline{\Om}} \mathbf{U} \cdot {\rm d} \mu_B(t)\dt.
\eeq
To calculate the term involved with the magnetic field, we choose $\mathbf{b}$ as a test function in (\ref{lbb5}) to deduce that
\beq\label{lbn7}
\left[ \int_{\Om}  \vB \cdot\mathbf{b}   \dx \right]_{t=0}^{t=\tau}
=
\int_0^{\tau}\int_{\Om} \Big( \vB\cdot \pdt \mathbf{b}-
\mathbf{E} \cdot
\Curl \mathbf{b} \Big) \dxdt.
\eeq

Combining (\ref{lbn4})--(\ref{lbn7}) with the balance of total energy (\ref{lbb7}), we obtain the \emph{relative energy inequality} as follows
\[
\left[\mathcal{E}\Big((\vr,\vu,\vB)\,\Big|\,(r,\mathbf{U},\mathbf{b})\Big)\right]_{t=0}^{t=\tau}
+\int_0^{\tau}\int_{\Om} \Big(  \mathbb{S}(\Grad\vu-\Grad \mathbf{U}):(\Grad \vu-\Grad \mathbf{U})
\Big)\dxdt
\]
\[
+ \int_0^{\tau}\int_{\Om} \left|  \mathbf{j} -\Curl \mathbf{b}\right|^2 \dxdt
+\int_{\overline{\Om} }{\rm d}\mathfrak{D}(\tau)
+ \int_0^{\tau}
  \int_{\overline{\Om}}  {\rm d}\mathfrak{C}
\]
\[
\leq -\int_0^{\tau}\int_{\Om} \Big ( \vr\vu\cdot \pdt \mathbf{U} + \vr\vu\otimes \vu: \Grad \mathbf{U}+p(\vr)\Div \mathbf{U}\Big ) \dxdt
\]
\[
+\int_0^{\tau}\int_{\Om} \Big ( \vr \mathbf{U}\cdot \pdt \mathbf{U}+ \vr\vu \cdot \Grad \mathbf{U} \cdot \mathbf{U}\Big )\dxdt
+\int_0^{\tau}\int_{\Om} \mathbb{S} (\Grad \mathbf{U}):
\left(
\Grad \mathbf{U}-\Grad \vu
\right)    \dxdt
\]
\[
+\int_0^{\tau}\int_{\Om}
\left[
\left( 1-\f{\vr}{r} \right)p'(r)\pdt r
-\vr\vu\cdot \f{p'(r)}{r} \Grad r
\right]
\dxdt
\]

\[
+ \int_0^{\tau}\int_{\Om}
\pdt \mathbf{b}  \cdot
(
\mathbf{b}-
  \vB
)
\dxdt
+
\int_0^{\tau}\int_{\Om}
\Curl \mathbf{b} \cdot (
\Curl \mathbf{b}-  \mathbf{j}
) \dxdt
\]
\[
-
\int_0^{\tau}\int_{\Om}
(\vu \times \vB)\cdot \Curl \mathbf{b} \dxdt
-
\int_0^{\tau}\int_{\Om}
(
\mathbf{j} \times
\mathbf{B}
) \cdot
\mathbf{U} \dxdt
\]
\beq\label{lbn8}
-
 \int_0^{\tau}\int_{\overline{\Om}} \Grad \mathbf{U} :{\rm d} \mu_c(t)\dt
 -
 \int_0^{\tau}\int_{\overline{\Om}} \mathbf{U} \cdot {\rm d} \mu_B(t)\dt,
\eeq
where we have used the identity  $\vE= \vj-\vu\times \vB$.

\subsection{Weak--strong uniqueness principle}\label{weak_str}

Let $(\widetilde{\vr},\widetilde{\vu},\widetilde{\vB},\widetilde{\mathbf{E}})$ be a classical solution to \eqref{pde}--\eqref{ini_c} starting from the smooth initial data $(\vr_0,\vu_0,\vB_0)$ with strictly positive $\vr_0$ and $\Div \vB_0=0$. Let $(\vr,\vu,\vB,\vE)$ be a dissipative weak solution to \eqref{pde}--(\ref{BCs}) emanating from the same initial data. It follows from (\ref{lbn8}) that
\[
\mathcal{E}\Big((\vr,\vu,\vB)\,\Big|\,(\widetilde{\vr},\widetilde{\vu},\widetilde{\vB})\Big)(\tau)
+\int_0^{\tau}\int_{\Om} \Big(  \mathbb{S}(\Grad\vu-\Grad \widetilde{\vu}):(\Grad \vu-\Grad \widetilde{\vu})
\Big)\dxdt
\]
\[
+\int_0^{\tau}\int_{\Om} \left|    \mathbf{j} -\Curl \widetilde{\vB}\right|^2\dxdt
+\int_{\overline{\Om} }{\rm d}\mathfrak{D}(\tau)
+ \int_0^{\tau}
  \int_{\overline{\Om}}  {\rm d}\mathfrak{C}
\]
\[
\leq -\int_0^{\tau}\int_{\Om} \Big ( \vr\vu\cdot \pdt \widetilde{\vu} + \vr\vu\otimes \vu: \Grad \widetilde{\vu}+p(\vr)\Div \widetilde{\vu}\Big ) \dxdt
\]
\[
+\int_0^{\tau}\int_{\Om} \Big ( \vr\widetilde{\vu}\cdot \pdt \widetilde{\vu}+ \vr\vu\cdot \widetilde{\vu}\cdot \Grad \widetilde{\vu} \Big )\dxdt
+\int_0^{\tau}\int_{\Om} \mathbb{S} (\Grad \widetilde{\vu}):
\left(
\Grad \widetilde{\vu}-\Grad \vu
\right)    \dxdt
\]
\[
+\int_0^{\tau}\int_{\Om}
\left[
\left( 1-\f{\vr}{\widetilde{\vr}} \right)p'(\widetilde{\vr})\pdt \widetilde{\vr}
-\vr\vu\cdot \f{p'(\widetilde{\vr})}{\widetilde{\vr}} \Grad \widetilde{\vr}
\right]
\dxdt
\]
\[
+ \int_0^{\tau}\int_{\Om}
\pdt \widetilde{\mathbf{B} } \cdot
(
\widetilde{\mathbf{B} }-
 \vB
)
\dxdt
+
\int_0^{\tau}\int_{\Om}
\Curl \widetilde{\mathbf{B} } \cdot (
\Curl \widetilde{\mathbf{B} }-  \mathbf{j}
)\dxdt
\]
\[
-
\int_0^{\tau}\int_{\Om}
(\vu \times \vB)\cdot \Curl \widetilde{\mathbf{B} } \dxdt
-
\int_0^{\tau}\int_{\Om}
(
\mathbf{j} \times
\mathbf{B}
) \cdot
\widetilde{\mathbf{u}} \dxdt
\]
\beq\label{lbn9}
-
 \int_0^{\tau}\int_{\overline{\Om}} \Grad \widetilde{\vu} :{\rm d} \mu_c(t)\dt
 -
 \int_0^{\tau}\int_{\overline{\Om}} \widetilde{\vu} \cdot {\rm d} \mu_B(t)\dt.
\eeq
By the compatibility conditions (\ref{lbb8})-(\ref{lbb8-2}),
\beq\label{lbn10}
\left|
 \int_0^{\tau}\int_{\overline{\Om}} \Grad \widetilde{\vu} :{\rm d} \mu_c(t)\dt
 +
 \int_0^{\tau}\int_{\overline{\Om}} \widetilde{\vu} \cdot {\rm d} \mu_B(t)\dt
\right|
\lesssim
 \epsilon \int_0^{\tau}
  \int_{\overline{\Om}}  {\rm d}\mathfrak{C}
  +\f{C}{\epsilon} \int_0^{\tau} \int_{\overline{\Om}}{\rm d}\mathfrak{D}(t)\dt,
\eeq
where $\epsilon>0$ is small enough. Observing that $(\widetilde{\vr},\widetilde{\vu},\widetilde{\vB},\widetilde{\mathbf{E}})$ solves \eqref{pde}--\eqref{ini_c} in the classical sense, i.e.,
\begin{equation}\label{lbn11}
\left\{\begin{aligned}
& \pdt \widetilde{\vr}+\Div (\widetilde{\vr}\, \widetilde{\vu})=0,\\
& \widetilde{\vr} \left(\pdt \widetilde{\vu}+\widetilde{\vu}\cdot \Grad \widetilde{\vu} \right)
+\Grad p(\widetilde{\vr})=\Div \mathbb{S}(\Grad \widetilde{\vu})+ \widetilde{\mathbf{j}}\times \widetilde{\vB},\\
& \widetilde{\mathbf{j}}=\Curl \widetilde{\vB},\\
& \pdt \widetilde{\vB}+\Curl  \widetilde{\mathbf{E}}=0 ,\\
& \Div \widetilde{\vB}=0, \\
\end{aligned}\right.
\end{equation}
with $\widetilde{\mathbf{j}}=\widetilde{\mathbf{E}}+\widetilde{\vu}\times \widetilde{\vB}$, we furthermore simplify the right-hand side of (\ref{lbn9}). Since this process is straightforward and similar to the compressible Navier--Stokes system (see \cite{FGSGW}), the details are omitted. Thus, 
\[
\mathcal{E}\Big((\vr,\vu,\vB)\,\Big|\,(\widetilde{\vr},\widetilde{\vu},\widetilde{\vB})\Big)(\tau)
+\int_0^{\tau}\int_{\Om} \Big(  \mathbb{S}(\Grad\vu-\Grad \widetilde{\vu}):(\Grad \vu-\Grad \widetilde{\vu})
\Big)\dxdt
\]
\[
+\int_0^{\tau}\int_{\Om} \left|    \mathbf{j} -\Curl \widetilde{\vB}\right|^2\dxdt
+\int_{\overline{\Om} }{\rm d}\mathfrak{D}(\tau)
+ \int_0^{\tau}
  \int_{\overline{\Om}}  {\rm d}\mathfrak{C}
\]
\[
\lesssim \int_0^{\tau}\int_{\Om}
\vr(\vu-\widetilde{\vu}) \cdot \Grad \widetilde{\vu}\cdot (\widetilde{\vu}-\vu)
\dxdt
+\int_0^{\tau}\int_{\Om} \mathbb{S} (\Grad \widetilde{\vu}):
\left(
\Grad \widetilde{\vu}-\Grad \vu
\right)    \dxdt
\]
\[
+\int_0^{\tau}\int_{\Om}
\vr(\widetilde{\vu}-\vu)    \cdot
\f{1}{\widetilde{\vr}} \Div \mathbb{S}(\Grad \widetilde{\vu})
\dxdt
-\int_0^{\tau}\int_{\Om}
p(\vr)-p(\widetilde{\vr}) -p'(\widetilde{\vr}) (\vr-\widetilde{\vr})  \Div \widetilde{\vu}
\dxdt
\]
\[
+\int_0^{\tau}\int_{\Om}
\vr(\widetilde{\vu}-\vu)   \cdot
\f{1}{\widetilde{\vr}} \left(\Curl \widetilde{\vB}\times \widetilde{\vB}\right)
\dxdt
\]
\[
+ \int_0^{\tau}\int_{\Om}
\pdt \widetilde{\mathbf{B} } \cdot
(
\widetilde{\mathbf{B} }-
  \vB
)
\dxdt
+
\int_0^{\tau}\int_{\Om}
\Curl \widetilde{\mathbf{B} } \cdot (
\Curl \widetilde{\mathbf{B} }- \mathbf{j}
)\dxdt
\]
\beq\label{lbn12}
-
\int_0^{\tau}\int_{\Om}
(\vu \times \vB) \cdot \Curl \widetilde{\mathbf{B} } \dxdt
-
\int_0^{\tau}\int_{\Om}
(
\mathbf{j} \times
\mathbf{B}
) \cdot
\widetilde{\mathbf{u}} \dxdt
+\int_0^{\tau} \int_{\overline{\Om}}{\rm d}\mathfrak{D}(t)\dt.
\eeq
Notice that the integrals involved with the magnetic field and the current density may be rewritten as, using (\ref{lbn11}),
\[
\int_0^{\tau}\int_{\Om}
\vr(\widetilde{\vu}-\vu)    \cdot
\f{1}{\widetilde{\vr}} \left(\Curl \widetilde{\vB}\times \widetilde{\vB}\right)
\dxdt
\]
\[
+ \int_0^{\tau}\int_{\Om}
\pdt \widetilde{\mathbf{B} } \cdot
(
\widetilde{\mathbf{B} }-
  \vB
)
\dxdt
+
\int_0^{\tau}\int_{\Om}
\Curl \widetilde{\mathbf{B} } \cdot (
\Curl \widetilde{\mathbf{B} }-  \mathbf{j}
)\dxdt
\]
\[
-
\int_0^{\tau}\int_{\Om}
(\vu \times \vB )\cdot \Curl \widetilde{\mathbf{B} } \dxdt
-
\int_0^{\tau}\int_{\Om}
(
\mathbf{j} \times
\mathbf{B} )
 \cdot
\widetilde{\mathbf{u}} \dxdt
\]
\[
=\int_0^{\tau}\int_{\Om}
(\vr-\widetilde{\vr})(\widetilde{\vu}-\vu)   \cdot
\f{1}{\widetilde{\vr}}\left( \Curl \widetilde{\vB}\times \widetilde{\vB} \right)
\dxdt
+\int_0^{\tau}\int_{\Om}
(\widetilde{\vu}-\vu)   \cdot
 \left( \Curl \widetilde{\vB}\times \widetilde{\vB} \right)
\dxdt
\]
\[
+\int_0^{\tau}\int_{\Om}
\Big(
\Curl  (\widetilde{\vu} \times \widetilde{\vB}) \cdot
(
\widetilde{\vB}-\vB
)
-(\vu \times \vB)\cdot \Curl  \widetilde{\vB}
\Big)
\dxdt
+\int_0^{\tau}\int_{\Om}
 \mathbf{j} \cdot
(\widetilde{\vu}\times \vB )
\dxdt
\]
\[
=\int_0^{\tau}\int_{\Om}
(\vr-\widetilde{\vr})(\widetilde{\vu}-\vu)    \cdot
\f{1}{\widetilde{\vr}}\left( \Curl \widetilde{\vB}\times \widetilde{\vB} \right)
\dxdt
 +
 \int_0^{\tau}\int_{\Om}
 \Big(  \mathbf{j} -\Curl \widetilde{\vB}\Big)
   \cdot
   \Big( \widetilde{\vu} \times (\vB-\widetilde{\vB})
   \Big )
   \dxdt
\]
\beq\label{lbn13}
 +\int_0^{\tau}\int_{\Om}
  \Curl  \widetilde{\vB} \cdot
 (\vu-\widetilde{\vu})\times(\widetilde{\vB}-\vB)
  \dxdt.
\eeq
Moreover, it holds that
\[
\left|
\int_0^{\tau}\int_{\Om}
 \Big(  \mathbf{j} -\Curl \widetilde{\vB}\Big)
   \cdot
   \Big( \widetilde{\vu} \times (\vB-\widetilde{\vB})
   \Big )
   \dxdt
\right|
\]
\beq\label{lbn14}
\lesssim
\epsilon \int_0^{\tau}\int_{\Om} \left| \mathbf{j} -\Curl \widetilde{\vB}\right|^2\dxdt
+\f{C}{\epsilon} \int_0^{\tau}\int_{\Om}
|\vB-\widetilde{\vB}|^2
\dxdt;
\eeq
\[
\left|
\int_0^{\tau}\int_{\Om}
  \Curl  \widetilde{\vB} \cdot
 (\vu-\widetilde{\vu})\times(\widetilde{\vB}-\vB)
  \dxdt
\right|
\]
\beq\label{lbn15}
\lesssim
\epsilon \int_0^{\tau}\int_{\Om}  |\vu-\widetilde{\vu}|^2 \dxdt
+\f{C}{\epsilon} \int_0^{\tau}\int_{\Om}
|\vB-\widetilde{\vB}|^2
\dxdt.
\eeq
Due to the generalized Korn-type inequality,
\beq\label{lbn16}
\int_0^{\tau}\int_{\Om}  |\vu-\widetilde{\vu}|^2 \dxdt
\lesssim
\int_0^{\tau}\int_{\Om} \Big(  \mathbb{S}(\Grad\vu-\Grad \widetilde{\vu}):(\Grad \vu-\Grad \widetilde{\vu})
\Big)\dxdt.
\eeq
It then follows from the isentropic law of pressure that
\[
\left|
\int_0^{\tau}\int_{\Om}
\Big( p(\vr)-p(\widetilde{\vr}) -p'(\widetilde{\vr}) (\vr-\widetilde{\vr}) \Big) \Div \widetilde{\vu}
\dxdt
\right|
\]
\beq\label{lbn17}
\lesssim
\int_0^{\tau}\int_{\Om}
\Big( H(\vr)-H(\widetilde{\vr}) -H'(\widetilde{\vr}) (\vr-\widetilde{\vr}) \Big)
\dxdt.
\eeq

Consequently, combining (\ref{lbn12})--(\ref{lbn17}) and choosing $\epsilon>0$ sufficiently small gives rise to
\[
\mathcal{E}\Big((\vr,\vu,\vB)\,\Big|\,(\widetilde{\vr},\widetilde{\vu},\widetilde{\vB})\Big)(\tau)
+\int_0^{\tau}\int_{\Om} \Big(  \mathbb{S}(\Grad\vu-\Grad \widetilde{\vu}):(\Grad \vu-\Grad \widetilde{\vu})
\Big)\dxdt
\]
\[
+\int_0^{\tau}\int_{\Om} \left|    \mathbf{j} -\Curl \widetilde{\vB}\right|^2\dxdt
+\int_{\overline{\Om} }{\rm d}\mathfrak{D}(\tau)
+ \int_0^{\tau}
  \int_{\overline{\Om}}  {\rm d}\mathfrak{C}
\]
\[
\lesssim
\int_0^{\tau}\int_{\Om} \mathbb{S} (\Grad \widetilde{\vu}):
\left(
\Grad \widetilde{\vu}-\Grad \vu
\right)    \dxdt
+\int_0^{\tau}\int_{\Om}
\vr(\widetilde{\vu}-\vu)    \cdot
\f{1}{\widetilde{\vr}} \Div \mathbb{S}(\Grad \widetilde{\vu})
\dxdt
\]
\[
+\int_0^{\tau}\int_{\Om}
(\vr-\widetilde{\vr})(\widetilde{\vu}-\vu)    \cdot
\f{1}{\widetilde{\vr}}\left( \Curl \widetilde{\vB}\times \widetilde{\vB} \right)
\dxdt
\]
\beq\label{lbn18}
+
\int_0^{\tau}
\mathcal{E}\Big((\vr,\vu,\vB)\,\Big|\,(\widetilde{\vr},\widetilde{\vu},\widetilde{\vB})\Big)(t)
\dt+
\int_0^{\tau} \int_{\overline{\Om}}{\rm d}\mathfrak{D}(t)\dt.
\eeq
Following \cite{FGSGW,FNS}, we estimate the remaining integrals as follows. Let $\chi$ be a cut-off function such that
\begin{equation}\label{lbn19}
\left\{\begin{aligned}
& \chi \in C _c^{\infty}((0,\infty)),\\
& 0\leq \chi \leq 1,\\
& \chi(\vr)=1 \text{ if }\vr \in [\inf{\widetilde{\vr}},\sup{\widetilde{\vr}}].\\
\end{aligned}\right.
\end{equation}
Thus we may write
\[
\left|
\int_0^{\tau}\int_{\Om}
(\vr-\widetilde{\vr})(\widetilde{\vu}-\vu)    \cdot
\f{1}{\widetilde{\vr}}\left( \Curl \widetilde{\vB}\times \widetilde{\vB} \right)
\dxdt
\right|
\]
\beq\label{lbn20}
\lesssim
\int_0^{\tau}\int_{\Om}
\chi(\vr)|\vr-\widetilde{\vr}||\widetilde{\vu}-\vu|
\dxdt
+
\int_0^{\tau}\int_{\Om}
(1-\chi(\vr))|\vr-\widetilde{\vr}||\widetilde{\vu}-\vu|
\dxdt.
\eeq
The first integral on the right-hand side of (\ref{lbn20}) is bounded through
\[
\int_0^{\tau}\int_{\Om}
\chi(\vr)|\vr-\widetilde{\vr}||\widetilde{\vu}-\vu|
\dxdt
\]
\[
\lesssim
\int_0^{\tau}\int_{\Om}
\f{1}{2}\vr|\widetilde{\vu}-\vu|^2
\dxdt
+
\int_0^{\tau}\int_{\Om}
 \f{1}{2}\f{\chi^2(\vr)}{\vr} |\vr-\widetilde{\vr}|^2
\dxdt
\]
\beq\label{lbn21}
\lesssim
\int_0^{\tau}
\mathcal{E}\Big((\vr,\vu,\vB)\,\Big|\,(\widetilde{\vr},\widetilde{\vu},\widetilde{\vB}) \Big)(t)
\dt.
\eeq
To estimate the second integral on the right-hand side of (\ref{lbn20}), we make a further decomposition, i.e.,
\[
1-\chi(\vr)=\chi_1(\vr)+\chi_2(\vr)
\]
such that
\[
\text{supp}\chi_1 \subset [0,\inf{\widetilde{\vr}}],\,\,
\text{supp}\chi_2 \subset [\sup{\widetilde{\vr}},\infty].
\]
It follows from (\ref{lbn16}) that
\[
\int_0^{\tau}\int_{\Om}
\chi_1(\vr)|\vr-\widetilde{\vr}||\widetilde{\vu}-\vu|
\dxdt
\]
\[
\lesssim
\epsilon
\int_0^{\tau}\int_{\Om}  |\vu-\widetilde{\vu}|^2 \dxdt
+\f{C}{\epsilon}
\int_0^{\tau}\int_{\Om}
\chi_1^2(\vr)|\vr-\widetilde{\vr}|^2
\dxdt
\]
\beq\label{lbn22}
\lesssim
 \epsilon\int_0^{\tau}\int_{\Om} \Big(  \mathbb{S}(\Grad\vu-\Grad \widetilde{\vu}):(\Grad \vu-\Grad \widetilde{\vu})
\Big)\dxdt
+\f{C}{\epsilon}
\int_0^{\tau}
\mathcal{E}\Big((\vr,\vu,\vB)\,\Big|\,(\widetilde{\vr},\widetilde{\vu},\widetilde{\vB}) \Big)(t)
\dt.
\eeq
Clearly,
\[
\int_0^{\tau}\int_{\Om}
\chi_2(\vr)|\vr-\widetilde{\vr}||\widetilde{\vu}-\vu|
\dxdt
\]
\[
\lesssim
\int_0^{\tau}\int_{\Om}
\chi_2(\vr)\vr|\widetilde{\vu}-\vu| ^2
\dxdt
+\int_0^{\tau}\int_{\Om}
\chi_2(\vr)\vr
\dxdt.
\]
\beq\label{lbn23}
\lesssim
\int_0^{\tau}
\mathcal{E}\Big((\vr,\vu,\vB)\,\Big|\,(\widetilde{\vr},\widetilde{\vu},\widetilde{\vB})\Big)(t)
\dt.
\eeq
Taking (\ref{lbn20})--(\ref{lbn23}) into account,
\[
\left|
\int_0^{\tau}\int_{\Om}
(\vr-\widetilde{\vr})(\widetilde{\vu}-\vu)   \cdot
\f{1}{\widetilde{\vr}}\left( \Curl \widetilde{\vB}\times \widetilde{\vB} \right)
\dxdt
\right|
\]
\beq\label{lbn24}
\lesssim
 \epsilon \int_0^{\tau}\int_{\Om} \Big(  \mathbb{S}(\Grad\vu-\Grad \widetilde{\vu}):(\Grad \vu-\Grad \widetilde{\vu})
\Big)\dxdt
+\f{C}{\epsilon}
\int_0^{\tau}
\mathcal{E}\Big((\vr,\vu,\vB)\,\Big|\,(\widetilde{\vr},\widetilde{\vu},\widetilde{\vB}) \Big)(t)
\dt.
\eeq
Finally, notice also that the first two integrals on the right-hand side of (\ref{lbn18}) are estimated as above upon observing that
\[
\int_0^{\tau}\int_{\Om} \mathbb{S} (\Grad \widetilde{\vu}):
\left(
\Grad \widetilde{\vu}-\Grad \vu
\right)    \dxdt
+\int_0^{\tau}\int_{\Om}
\vr(\widetilde{\vu}-\vu)   \cdot
\f{1}{\widetilde{\vr}} \Div \mathbb{S}(\Grad \widetilde{\vu})
\dxdt
\]
\beq\label{lbn25}
=
\int_0^{\tau}\int_{\Om}
(\vr-\widetilde{\vr})(\widetilde{\vu}-\vu)   \cdot
\f{1}{\widetilde{\vr}}
\Div \mathbb{S}(\Grad \widetilde{\vu})
\dxdt.
\eeq
Combining (\ref{lbn18}), (\ref{lbn24})--(\ref{lbn25}) and \emph{fixing} $\epsilon>0$ sufficiently small shows that
\[
\mathcal{E}\Big((\vr,\vu,\vB)\,\Big|\,(\widetilde{\vr},\widetilde{\vu},\widetilde{\vB}) \Big)(\tau)
+
\int_0^{\tau}\int_{\Om} \Big(  \mathbb{S}(\Grad\vu-\Grad \widetilde{\vu}):(\Grad \vu-\Grad \widetilde{\vu})
\Big)\dxdt
\]
\[
+\int_0^{\tau}\int_{\Om} \left|    \mathbf{j} -\Curl \widetilde{\vB}\right|^2\dxdt
+\int_{\overline{\Om} }{\rm d}\mathfrak{D}(\tau)
+ \int_0^{\tau}
  \int_{\overline{\Om}}  {\rm d}\mathfrak{C}
\]
\beq\label{lbn26}
\lesssim
\int_0^{\tau}
\mathcal{E}\Big((\vr,\vu,\vB)\,\Big|\,(\widetilde{\vr},\widetilde{\vu},\widetilde{\vB}) \Big)(t)
\dt+
\int_0^{\tau} \int_{\overline{\Om}}{\rm d}\mathfrak{D}(t)\dt.
\eeq
We conclude from Gronwall's inequality that
\[
\vr=\widetilde{\vr},\,\,\,\,\vu=\widetilde{\vu},\,\,\,\,\vB=\widetilde{\vB},\,\,\,\,\vE=\widetilde{\vE}\,\,\,\,
\text{      in    } (0,T)\times \Om,
\]
\[
\mu_c=\mathbf{0},\,\,\mu_B=0,\,\, \mathfrak{D}=\mathbf{0},\,\, \mathfrak{C}=0.
\]
The proof of Theorem \ref{Th1} is thus finished.           $\Box$

\section{Convergence}\label{sec_con}
In this section we first prove Theorem~\ref{Th2} for the convergence of a consistent approximation towards a DW solution and a classical solution. Further, we discuss its application in the convergence analysis of numerical solutions. 
\subsection{Convergence of a consistent approximation} 
\begin{proof}[Proof of Item 1 of Theorem~\ref{Th2}.]
The proof can be done in the same way as our previous work (\cite[Theorem 2.5]{LiShe_MHD1}), see also \cite{FeLMMiSh}. Here we list the main idea in the proof.  
\begin{itemize}
\item{\bf Step 1: convergence towards a DW solution.}
We deduce from the energy stability \eqref{es} for suitable subsequences that
\[
\vrh \rightarrow \vr \ \mbox{weakly-(*) in}\ L^\infty(0,T; L^\gamma(\Omega)), \vr \geq 0 ,
\]
\[
 \vuh,  \Pi_h{\vuh}\rightarrow \bfu \ \mbox{weakly in}\ L^2((0,T) \times \Omega; \R^d),\mbox{where}\,\,  \bu   \in L^2(0,T; W^{1,2}(\Omega;\R^d)), 
 \]
 \[
 \text{ and } \bu \in  L^2(0,T; W_0^{1,2}(\Omega;\R^d)) \text{ in case of no-slip boundary conditions},
\]
\[
 \Gradh \vuh \to \Grad \bu \ \mbox{weakly in} \ L^2((0,T) \times \Omega;
\R^{d \times d}), 
\]
\[
 \vrh  \Pi_h{\vuh}  \rightarrow \Ov{\vr \vu}  \,\, \mbox{weakly-(*) in}\ L^\infty(0,T; L^{\frac{2\gamma}{\gamma + 1}}(\Omega; \R^d)),
\]
\[
\vBh \rightarrow \vB \ \mbox{weakly-(*) in}\ L^\infty(0,T; L^2(\Omega;\R^d)),
\]
\[
\vEh \rightarrow  \vE \ \mbox{weakly-(*) in}\ L^2(0,T; L^{\frac32}(\Omega;\R^d)),
\]
\[
\vrh \Pi_h{\vuh} \otimes \vuh + p(\vrh) \mathbb{I} \rightarrow  \Ov{ 1_{\vr>0} \frac{\vm \otimes \vm}{\vr} + p(\vr) \mathbb{I} }
 \ \mbox{weakly-(*) in} \ 
  L^\infty(0,T; \mathcal{M}(\overline{\Omega};\R^{d\times d}_{sym})),
\]
\[
\vjh \times \vBh  \rightarrow  \Ov{ \vj \times \vB }
 \ \mbox{weakly-(*) in} \   L^2(0,T; \mathcal{M}(\overline{\Omega};\R^d)) ,
 \]
 \[
 \S(\Gradh \vuh) : \Gradh \vuh + |\vjh|^2   \rightarrow \overline{\mathbb{S}(\Grad \vu):\Grad \vu + |\vj|^2} \mbox{ in }\mathcal{M}^{+}([0,T]\times \overline{\Om}),
% \\
%& |\vjh|^2  \to \overline{ |\vj|^2 }\mbox{ in }\mathcal{M}^{+}([0,T]\times \overline{\Om}),
\]
\[
 \frac{1}{2} \vrh \abs{\Pi_h  \vuh}^2  + H(\vrh) +\frac{1}{2} \abs{\vBh}^2 \rightarrow \overline{  \f{1}{2}\vr |\vu|^2+H(\vr)+\f{1}{2}|\mathbf{B}|^2 }\mbox{  weakly-(*) in }L^{\infty}(0,T; \mathcal{M}^{+}(\overline{\Om})),
\]
where $\vm =\vr \vu$.  It holds analogous to \cite[Lemma 3.7]{AnFeNo} that 
\[ \Ov{\vr \vu} = \vr \vu, \quad \Ov{ \vu \times \vB} =  \vu \times \vB
\]
Moreover, we define the following concentration/oscillation defect measures:
\[
\mu_c:=\overline{      1_{\vr>0}\f{\mathbf{m}\otimes \mathbf{m}}{\vr} +p(\vr)\mathbb{I}    }
-\left(
 1_{\vr>0}\f{\mathbf{m}\otimes \mathbf{m}}{\vr} +p(\vr)\mathbb{I}
\right),
\]
\[
\mu_B:=\overline{ \vj \times \vB }
-(  \vj \times \vB ),
\]
\[
\mathfrak{D}:=\overline{  \f{1}{2}\vr |\vu|^2+H(\vr)+\f{1}{2}|\mathbf{B}|^2 }
-\left(\f{1}{2}\vr |\vu|^2+H(\vr)+\f{1}{2}|\mathbf{B}|^2\right),
\]
\[
\mathfrak{C}=\overline{\mathbb{S}(\Grad \vu):\Grad \vu  +   |\vj|^2 } - \Big( \mathbb{S}(\Grad \vu):\Grad \vu+  |\vj|^2  \Big); 
\]
whence it is easy to see that
\begin{equation}\label{comp}
|\mu_c|  \aleq  \mathfrak{D}, \quad |\mu_B|  \aleq \frac{1}{\epsilon}   \mathfrak{D} +  \epsilon  \mathfrak{C} ,  \, \text{ for any }\epsilon >0. 
\end{equation}

Consequently, passing to the limit $h \to 0$ for the consistency formulation \eqref{cP1}--\eqref{cP4} and the energy stability \eqref{es}, we find that the resulting limit formulae exactly coincide with \eqref{lbb3} -- \eqref{lbb7} in the definition of DW solution. Moreover, the compatibility conditions \eqref{lbb8} hold due to \eqref{comp}. Item 1 of Theorem~\ref{Th2} is thus proved.
\item{\bf Step 2: convergence towards a classical solution.}  
Combining Theorem~\ref{Th1} and Item 1 of Theorem~\ref{Th2}  we immediately get Item 2 of Theorem~\ref{Th2}, that is the convergence of the consistent approximations towards the classical solution emanating from the same initial data.
\end{itemize}
\end{proof}

%\begin{Remark}
From Theorem~\ref{Th2} we conclude that any numerical solution converges to a DW solution and a classical solution (on its lifespan) if the numerical solution is a consistent approximation. 
%Recalling the definition of the consistent approximation stated in Definition~\ref{def_ca}, we realize that the convergence of a numerical solution is equivalent to show that it is stable and consistent in general. 
%\end{Remark}
In the rest of this section, we show  the application of Theorem~\ref{Th2} in the convergence analysis of numerical solutions by two examples. 

\subsection{Example I}
In this example we construct a numerical solution that falls in the class of consistent approximation. 
To this end, we propose a mixed finite volume (FV) -- finite element (FE) method fine adapted from the compressible Navier-Stokes solver of Karper~\cite{Karper} and the electric-magnetic solver of Hu et al.~\cite{HMX}. 
To begin, we introduce the necessary notations. 
\paragraph{Mesh.} Let $\Oh$ be a regular triangulation of the bounded domain $\Omega$,
$\mcE$ be the set of all $(d-1)$-dimensional faces, $\mcEe=\mcE \cap \pd \Omega$ be the exterior faces, $\mcEi=\mcE \setminus \mcEe$ be the interior faces, and $\facesK$ be the set of all faces of an arbitrary element $K$.
We denote $\sigma=K|L$ as the common face of two neighbouring elements $K$ and $L$.
Further, we denote $\bfn_\sigma$ as the outer normal of a face $\sigma \in \faces$ and $\bfn_{\sigma,K}$ as the unit normal vector pointing outwards $K$ if $\sigma \in \facesK$.
The size of the mesh (maximal diameter of all elements) is supposed to be a positive parameter $h<1$.

We suppose $\TS \approx h$ and denote
 $t^k= k\TS$ for $k=1,\ldots,N_T(=T/\TS)$.

\paragraph{Function spaces.}\label{sec:spaces}
First, we define a discrete function space of piecewise constants  $Q_h$, 
and denote $\CR, \RT, \ND$ as the piecewise linear Crouzeix--Raviart,  $\Hcurl$-N\'{e}d\'{e}lec and  $\Hdiv$-N\'{e}d\'{e}lec type element spaces, respectively. These function spaces read
 \begin{equation*}
  Q_h \equiv \{ v \in L^2(\Omega): \ v \vert_K \in \mcP^1_0, \ \forall\ K \in \Oh \},
 \end{equation*}
 \begin{equation*}
 \begin{aligned}
 V_h \equiv \Big\{ \vv \in L^2(\Omega):\;  \vv \vert_K \in \mcP^d_1(K), \forall\,  K \in \Oh; \, 
\intG{ \jump{\vv}}  =0, \forall\,  \sigma \in \mcEi
 \Big\},
 \end{aligned}\end{equation*}
  \begin{equation*}\begin{aligned}
  \RT \equiv \Big\{ \vv, \Div \vv \in L^2(\Omega): \;  \vv \vert_K \in \mcP_0^d\oplus\mcP_0^1 \bfx, \; \forall\;  K \in \Oh;\;
  \intG{ \jump{\vv\cdot \bfn} } =0, \; \forall\;  \sigma \in \mcEi   \Big\}, 
 \end{aligned}\end{equation*}
 \begin{equation*}\begin{aligned}
  \ND \equiv \Big\{ \vv, \Curl \vv \in L^2(\Omega): \;  \vv \vert_K \in \mcP_0^d\oplus\mcP_0^1 \bfx, \; \forall\;  K \in \Oh;\, 
  \intG {\jump{\vv \times \bfn}} =0, \; \forall\;  \sigma \in \mcEi
   \Big\}, 
 \end{aligned}\end{equation*}
 where $\mcP^m_n$ denotes the space of polynomials of degree not greater than $n$ for $m$-dimensional vector valued functions ($m=1$ for scalar functions). 
Note that $\RT \subset \HcurlO$ and $\ND \subset \HdivO$. 
The interpolation operators associated to the function spaces $\CR$, $\RT$, and $\ND$ are given by
\begin{equation}\label{proj1}
\begin{aligned}
\PiV :\ W^{1,2}(\Omega) \rightarrow  \CR, \quad
\PiR :\ W^{1,2}(\Omega) \rightarrow \RT, \quad
\PiN :\ W^{1,2}(\Omega) \rightarrow \ND.
\end{aligned}
\end{equation}
%%%%%%%%%%%%%%%%%%%%%%%%%%%%%%%%%%%%%%%
Further, we denote $\CRz = \{ \vv \in \CR : \,   \intG \vv  =0, \forall\,  \sigma \in \mcEe  \}$, 
$  \RTz = \{ \vv \in \RT : \,     \intG {\vv \cdot \bfn} =0, \; \forall\;  \sigma \in \mcEe  \}$ and 
$ \NDz = \{ \vv \in \ND : \,   \intG {\vv\times \bfn} =0, \; \forall\;  \sigma \in \mcEe \}$.

Next, we denote a generic discrete function $r_h$ at time $t^k$ by $r_h^k$ and define $L_{\TS}(0,T;Y)$ as the space of piecewise constant in time functions, i.e., for any $v_h \in L_{\TS}(0,T;Y)$ it means
\[ v_h(t,\cdot) =v_h^0 \mbox{ for } t \leq 0,\ v_h(t,\cdot)=v_h^k \mbox{ for } t\in ((k-1)\TS,k\TS],\ k=1,2,\ldots,N_T,
\]
where $Y \in \{Q_h, \CR, \RT,\ND\}$.
Then we approximate the unknown variables in the following functional spaces
\[ \vrh \in L_{\TS}(0,T;Q_h), \;
\vuh \in  L_{\TS}(0,T;\CR),\;
\vBh \in  L_{\TS}(0,T;\RT),\;
\vEh \in  L_{\TS}(0,T;\ND).
\]
Further, we define for simplicity $\FS=Q_h \times \CR \times \RT \times \ND$ and  $\FSz=Q_h \times \CRz \times \RTz \times \NDz $.

\paragraph{Some discrete operators.}
We define the discrete time derivative by the backward Euler method
\[
 D_t v_h = \frac{v_h (t) - v_h^{\triangleleft}(t)}{\TS}  \mbox{ for } t\in(0,T) \quad \mbox{ with } \quad   v_h^{\triangleleft}(t) =  v_h(t - \Delta t) .
\]
For any piecewise continuous function $f$, we define its trace on a generic edge %$\sigma \in \mcE$
as
\begin{equation*}
\begin{aligned}
f^{\rm in}|_{\sigma} = \lim_{\delta \rightarrow 0^+} f(\xx -\delta \nG ), \ \forall\  \sigma \in \mcE,
\qquad
f^{\rm out}|_{\sigma} = \lim_{\delta \rightarrow 0^+} f(\xx +\delta \nG ), \ \forall\  \sigma \in \mcEi.
\end{aligned}
\end{equation*}
Note that $f^{\rm out}|_{\mcEe}$ is determined by the boundary condition.
Further, we define the jump and average operators at an edge $\sigma \in \mcE$ as
\begin{equation}\label{op_diff}
\jump{f}_{\sigma} = f^{\rm out} - f^{\rm in} \mbox{ and }
\avg{f} = \frac{f^{\rm out} + f^{\rm in}}{2} ,
 %\ \forall \; \sigma \in \mcE,
%\quad  \left. \avc{f}\right|_K = \frac{1}{\abs{K}} \intK f  \ \forall\ K \in \Oh.
\end{equation}
respectively, and denote $\Pi_h = \PiQ$ as the element-wise constant projection, where 
\begin{equation}\label{op_avgk}
\PiQ: L^1(\Omega) \to Q_h. \quad   \left. \PiQ f\right|_K =  \left. \avc{f}\right|_K \equiv \frac{1}{\abs{K}} \intK f   , \ \forall\ K \in \Oh.
\end{equation}
Next,  we introduce the upwind flux for any function $r_h \in Q_h$ at a generic face  $\sigma \in \facesint$
\begin{align*}\label{Up}
\Up [r_h, \vuh]   =r_h^{\rm up}\us
=r_h^{\rm in} [\us]^+ + r_h^{\rm out} [\us]^-,
%= \avg{r} \ \Ov{\vv} \cdot \vc{n} - \frac{1}{2} |\Ov{\vv} \cdot \vc{n}| \jump{r},
\end{align*}
where $\vuh \in \CR$ is the velocity field and %, $\us = \frac{1}{|\sigma|} \intG{\vu} \cdot \vc{n}_\sigma$ and
\begin{equation*}
\us = \frac{1}{|\sigma|} \intG{\vuh} \cdot \vc{n}_\sigma, \quad
[f]^{\pm} = \frac{f \pm |f| }{2} \quad \mbox{and} \quad
r_h^{\rm up} =
\begin{cases}
 r_h^{\rm in} & \mbox{if } \ \us \geq 0, \\
r_h^{\rm out} & \mbox{if } \ \us < 0.
\end{cases}
\end{equation*}
Furthermore, we consider a diffusive numerical flux function of the following form for $\eps>-1$
\begin{equation}\label{num_flux}
\begin{aligned}
\Fup(r_h,\vuh)
=\Up[r_h, \vuh] -  h^\eps \jump{ r_h }
.% = \avg{r} \ \Ov{\vv} \cdot \vc{n}  - \left( h^\eps +   \frac12 \abs{\Ov{\vv} \cdot \vc{n}} \right) \jump{r} .
\end{aligned}
 \end{equation}
It is easy to check for any $\vrh \in Q_h$ and $\vuh \in \CR$ that
\begin{equation}\label{fluxes}
\Fup(\vrh \avc{\vuh},\vuh) \jump{\avc{\vuh}} -  \Fup(\vrh ,\vuh) \jump{ \frac{|\avc{\vuh}|^2}{2}}
= - \sum_{ \sigma \in \facesint } \intG{ \left(\frac12 \vrh^{\up} \abs{\us}
+h^\eps \avg{\vrh}   \right)\abs{\jump{\avc{\vuh}}}^2 }  .
%- h^\eps \sum_{ \sigma \in \facesint } \intG{ \avg{\vr} \abs{\jump{\avc{\vu}}}^2 }
\end{equation}

For simplicity, we denote
$\co{a}{b} =\left[\min(a,b), \max(a,b) \right]$
and write
$ a \aleq b \mbox{ if } a \le cb$ if $c$ is a positive constant that is independent of the mesh size and time step used in the scheme.
We shall frequently use the abbreviation $\norm{\cdot}_{L^p}$ and $\norm{\cdot}_{L^pL^q}$  for $\norm{\cdot}_{L^p(\Omega)}$ and $\norm{\cdot}_{L^p\left(0,T;L^q(\Omega)\right)}$, respectively.

\paragraph{The numerical method.}
Using the above notation we propose a novel mixed FV -- FE method to approximate system \eqref{pde}--\eqref{ini_c}.

\begin{tcolorbox}
{\bf Scheme-I.} Let the pressure satisfy \eqref{plaw} with $\gamma>1$. 
Given the initial values  \eqref{ini_c} we set $(\vrh^0,\vuh^0, \vBh^0) =(\PiQ \vr_0, \PiV \vu_0, \PiN \vB_0)$ and  seek $(\vrh,\vuh, \vBh, \vEh) \in L_{\TS}(0, T;\FSz) $ such that
\begin{subequations}\label{scheme}
\begin{equation}\label{scheme_D}
\intO{ D_t \vrh  \phi_h } - \sum_{ \sigma \in \facesint } \intG{  \Fup(\vrh  ,\vuh  )
\jump{\phi_h}   } = 0 \quad \mbox{for all } \phi_h \in Q_h;
\end{equation}
\begin{equation}\label{scheme_M}
\begin{aligned}
  \intO{ D_t (\vrh  \auh) \cdot  \vh   } -
  \sum_{ \sigma \in \facesint } \intG{  \Fup(\vrh   \auh,\vuh  ) \cdot
\jump{\avh}   }
+ \mu \intO{  \Gradh \vuh  : \Gradh \vh  } \\ +  \intO{( \nu \Divh \vuh  - p_h ) \Divh \vh  }
 - \intO{ ( \vjh   \times  \vBh^\triangleleft ) \cdot \vh}
 =0  \quad \mbox{for all } \vh \in \CR ;
\end{aligned}
\end{equation}
\begin{equation}\label{scheme_B}
\intO{ \left( D_t \vBh \cdot \vCh + \vEh \cdot \Curlh \vCh \right)}=0
 \quad \mbox{for all } \vCh \in \RT .
\end{equation}
\begin{equation}\label{scheme_E}
\intO{\left( \vjh \cdot \bfpsi_h -  \vBh \cdot \Curlh \bfpsi_h \right) } =0  \quad \mbox{for all } \bfpsi_h \in \ND ;
\end{equation}
\end{subequations}
where $\nu = \frac{d-2}{d}\mu+\lambda$, $ \vjh = \vEh  + \vuh \times \vBh^\triangleleft$, and 
the discrete operators $\Divh, \Gradh$ and $\Curlh$ are the same as the continuous case on each element. 
%Moreover, the artificial diffusion parameter $\eps$ follows
%\begin{equation}\label{eps}
%\eps >0 \mbox{ if } \gamma \geq 2 \quad \mbox{ and }\quad   \eps \in(0, 2 \gamma-1 -d/3) \mbox{ if } \gamma \in(4d /(1+3d),2).
%\end{equation}
\end{tcolorbox}

Before going to the proof of the main results, we show some nice properties of the scheme.
\begin{Lemma}[Existence, mass conservation, renormalized continuity, positivity, divergence free]\label{lem_b}
\hfill
\begin{enumerate}
\item {\bf Existence of a numerical solution.}
There exists at least one solution to  {\bf Scheme-I}  \eqref{scheme}.
\item {\bf Mass conservation.} 
{\bf Scheme-I} preserves the conservation of the total mass.
\[
\intO{ \vrh (t) } = \intO{ \vrh (0) } = \intO{ \vr_0 } , \quad \forall\;  t \in[0,T].
\]
\item {\bf Renormalized continuity equation.}
Let $(\vrh,\vuh)\in  Q_h \times \CR$ satisfy the discrete continuity equation \eqref{scheme_D} and $b=b(\vr)\in C^2(0,\infty)$. Then the discrete continuity equation \eqref{scheme_D} can be renormalized in the sense that
\begin{equation}\label{r1}
\begin{aligned}
&\intO{\left(D_t b(\vrh ) -\big(\vrh  b'(\vrh )-b(\vrh )\big) \Divh \vuh  \right) }
\\&=
- \frac{\TS}2 \intO{ b''(\xi)|D_t \vrh |^2  }
- \sum_{ \sigma \in \facesint } \intG{ b''(\zeta) \jump{  \vrh  } ^2 \left(h^\eps  +  \frac12 | \us | \right) }.
\end{aligned}
\end{equation}
where $\xi \in \co{\vrh^{\triangleleft}}{\vrh }$ and  $\zeta \in \co{\vrh^{\rm in}}{\vrh^{\rm out}}$.
\item {\bf Positivity of the density.} Let $\vr_0>0$. Then any solution of {\bf Scheme-I}  \eqref{scheme} satisfies $\vrh(t)>0$ for $t\in [0,T]$.
\item {\bf Divergence free of magnetic field.}
Let $\Div \vB_0 =0$. Then {\bf Scheme-I} preserves divergence free of magnetic field exactly, meaning that
$ \Divh \vBh(t)  =0$ for all $t \in [0,T]$. 
\end{enumerate}
\end{Lemma}
\begin{proof}
\begin{itemize}
\item
The existence of a numerical solution to {\bf Scheme-I} \eqref{scheme} can be analogously proven as \cite{Ding} via the theorem of topological degree. %, see also similar results in \cite{SS_FSI}.
\item
Taking $\phi_h \equiv 1$ in the equation of continuity \eqref{scheme_D} immediately yields the mass conservation.
\item
We refer to~\cite[Lemma 4.1]{Karper} for the proof of renormalized continuity equation.
\item
Concerning the positivity of density, we refer to \cite[Lemma 3.2]{HS_MAC} for the proof, see also \cite{Karlsen,FeLMMiSh}.
\item  
By setting $\vCh = D_t \vBh + \Curlh \vEh \in \RT$ in the discrete problem \eqref{scheme_B} we obtain
\[\sum_{K \in \Oh} \int_K { \abs{D_t \vBh + \Curlh \vEh}^2}\dx=0,
\]
which implies $D_t \vBh + \Curlh \vEh=0$ for all $x\in K \in \Oh$. Applying a divergence operator to the equality and noticing that $\Divh (\Curlh \vEh) =0 $ we derive
\[ \Divh \vBh =  \Divh \vBh^{\triangleleft} = \cdots = \Divh \vBh^0 =0,
\]
which completes the proof. 
\end{itemize}
\end{proof}
%%%%%%%%%%%%%%%%

\subsubsection{Energy stability}
In this subsection, we show the energy stability of  scheme~\eqref{scheme} in the sense of  Definition~\ref{def_ca}. More precisely, we have the following energy stability. 

 \begin{Theorem}[Stability of {\bf Scheme-I}]\label{Tm1}\hspace{1em}\newline
 Let $(\vrh,\vuh, \vEh,\vBh)$ be a solution to {\bf Scheme-I} with $\gamma >1$. Then there exist $\xi \in \co{\vrh^{\triangleleft}}{\vrh }$ and  $\zeta \in \co{\vrh^{\rm in}}{\vrh^{\rm out}}$ for any $\sigma \in \facesint$ such that
 \begin{equation} \label{ke}
 \begin{split}
  D_t & \intO{ \left(\frac{1}{2} \vrh  \abs{\auh }^2  + \Hc(\vrh )
 +\frac{1}{2} \abs{\vBh }^2 \right)  }
 +  \mu \norm{\Gradh \vuh }_{L^2}^2  + \nu \norm{ \Divh \vuh }_{L^2}^2
 +  \norm{\vjh }_{L^2}^2
 \\ =&
 - \frac{\TS}{2} \intO{ \vrh^{\triangleleft}\abs{D_t \vuh }^2   }
 - \frac{ \TS}{2} \intO{ \abs{D_t \vBh }^2   }
 - \frac{\TS}2 \intO{ \Hc''(\xi)\abs{D_t \vrh }^2  }
 \\&
 - \sum_{ \sigma \in \facesint } \intG{ \left( \vrh^{\up} \frac{\abs{\us} }{2}
+h^\eps \avg{\vrh}   \right)\abs{\jump{\avc{\vuh}}}^2 }
% - h^\eps  \sum_{ \sigma \in \facesint } \intG{  \Ov{\vrh } \jump{ \vuh }^2  }
 -  \sum_{ \sigma \in \facesint } \intG{ \Hc''(\zeta) \jump{  \vrh  } ^2 \left(h^\eps  +\frac{\abs{\us} }{2} \right)}
  \leq 0.
 \end{split}
 \end{equation}
% where $\vjh= \vEh+ \vuh \times \vBh$ and $\nu = \frac{d-2}{d}\mu +\lambda$. 
\end{Theorem}
\begin{proof}
First, summing up \eqref{scheme_D} and \eqref{scheme_M} with the test functions $\phi_h = -\frac{\abs{\av{\vuh}}^2}{2}$ and $\bfphi_h = \vuh $  implies the discrete kinetic energy balance
%, see~\cite[Proposition 4.2]{Karper}, see also \cite[Section 4.2]{FKM}
\begin{equation*}%\label{k1}
\begin{aligned}
&\intO{D_t \left(\frac12 \vrh \abs{\av{\vuh}}^2 \right) } +  \mu   \norm{\Gradh \vuh}_{L^2}^2  + \nu \norm{ \Divh \vuh}_{L^2}^2 - \intO{ p_h \Divh \vuh }
\\& \quad
+  \frac{\TS}{2} \intO{ \vrh^{\triangleleft} \abs{D_t \vuh}^2}
+ h^\eps  \sum_{ \sigma \in \facesint } \intSh{  \avg{\vrh} \jump{ \vuh}^2  }
\\ &
= \intO{\vjh \times \vBh^\triangleleft \cdot \vuh }
=- \intO{\vjh  \cdot (\vuh \times \vBh^\triangleleft ) }
=- \intO{\vjh  \cdot (\vjh -\vEh ) }
.
\end{aligned}
\end{equation*}
Next, by setting $\vCh =\vBh$ in \eqref{scheme_B} and $\bfpsi_h=\vEh$ in \eqref{scheme_E}, respectively, we derive
\begin{equation*}%\label{k2}
\begin{aligned}
\intO{\vjh \cdot \vEh }
&=   \intO{\vBh \cdot \Curlh \vEh }
=  \intO{\vEh \cdot \Curlh \vBh }
= - \intO{D_t \vBh \cdot \vBh}
\\&
= - \frac{1}{2} \intO{\left( D_t \abs{\vBh}^2 + \TS \abs{D_t \vBh}^2\right) } .
\end{aligned}
\end{equation*}
Further, setting $b=\Hc(\vrh)$ in~\eqref{r1} and noticing the equality $\vr \Hc'(\vr) - H(\vr) =p(\vr)$ we observe the discrete internal energy balance
\begin{equation*}%\label{r1}
\intO{ \left( D_t \Hc(\vrh)  - p_h \Divh \vuh  \right) }
=
- \frac{\TS}2 \intO{\Hc''(\xi)|D_t \vrh|^2  }
- \frac12 \sum_{ \sigma \in \facesint } \intSh{ \Hc''(\zeta) \jump{  \vrh } ^2 \left(h^\eps  + | \us | \right) },
\end{equation*}
where $\xi \in \co{\vrh^{\triangleleft}}{\vrh}$ and  $\zeta \in \co{\vrh^{\rm in}}{\vrh^{\rm out}}$  are given in the third item of Lemma~\ref{lem_b}. Here, let us point out that $\Hc''(\vr)>0$ for all $\vr>0$ provided $\gamma>1$.

Finally, summing up the above three formulae completes the proof. 
\end{proof}

\paragraph{Uniform bounds.}
The {\sl a priori} estimates is a consequence of the energy estimates \eqref{ke}.
\begin{Lemma}[A priori estimates]\label{ests}
Let $(\vrh,\vuh,\vEh,\vBh)$ be a solution of {\bf Scheme-I} \eqref{scheme}. 
Then the following estimates hold
\begin{subequations}\label{ests1}
\begin{align}
\label{est_ener}
\norm{\vrh \abs{\auh}^2}_{L^{\infty}L^1}  \lesssim 1,\quad
\norm{p_h}_{L^{\infty}L^1}, \norm{\vrh}_{L^{\infty}L^\gamma}  \lesssim 1, \quad
\norm{\Gradh \vuh}_{L^2 L^2}   \lesssim 1,  \quad
\norm{\Divh \vuh}_{L^{2}L^2}     \lesssim 1,
% \\
% \label{estrjump}
% h^\eps \int_0^T \sum_{ \sigma \in \facesint } \intSh{  \avg{\vrh} \jump{ \vuh}^2}\dt  \lesssim 1, \quad
% \int_0^T  \sum_{ \sigma \in \facesint } \intSh{ \Hc''(\zeta) \jump{  \vrh } ^2 (h^\eps + | \Ov{\vuh} \cdot \vc{n} |) } \dt   \lesssim 1,
\\
\label{est_u}
\norm{\vuh}_{L^{2}L^6}  \lesssim 1,\quad
\norm{\vrh \av{\vuh} }_{L^\infty L^{\frac{2\gamma}{\gamma+1}}}  \lesssim 1,\quad
%\norm{\vrh\vuh}_{L^2 L^{\frac{6\gamma}{\gamma+6}}}  \lesssim 1,
\\
\norm{\vBh}_{L^{\infty}L^2}  \lesssim 1, \quad
\norm{\vjh}_{L^{2}L^2}  \lesssim 1, \quad
\norm{\vjh\times \vBh}_{L^{2}L^1}  \lesssim 1, \quad
\norm{\vEh}_{L^{2}L^{3/2}}  \lesssim 1. \label{est_B}
\end{align}
\end{subequations}
where $\zeta \in \co{\vr_K}{\vr_L}$ for any $\sigma=K|L \in \facesint$.
\end{Lemma}
\begin{proof}
Firstly, the first line \eqref{est_ener} is obvious from the  energy estimates \eqref{ke}.
Secondly, applying the Sobolev-Poincar\'{e} inequality and noticing the bounds of velocity gradient in \eqref{est_ener} imply the first estimates of \eqref{est_u}. Further, by H\"older's inequality, we derive the second estimate of \eqref{est_u}, i.e.,
\[ \norm{\vrh\auh}_{L^\infty L^{\frac{2\gamma}{\gamma+1}}}  =
\norm{\sqrt{\vrh} \sqrt{\vrh}\auh}_{L^\infty L^{\frac{2\gamma}{\gamma+1}}} \lesssim
\norm{\sqrt{\vrh}}_{L^\infty L^{2\gamma}}  \norm{\sqrt{\vrh}\auh}_{L^\infty L^2}
= \norm{\vrh}_{L^\infty L^{\gamma}}^{1/2}  \norm{\vrh \abs{\auh}^2}_{L^\infty L^1} ^{1/2}
\lesssim 1,
\]
%\[ \norm{\vrh\vuh}_{L^2 L^{\frac{6\gamma}{\gamma+6}}}
%\lesssim \norm{\vrh}_{L^\infty L^{\gamma}}  \norm{\vuh}_{L^2 L^6}
%\lesssim 1.
%\]
Next, the energy estimate \eqref{ke} directly implies the first two estimates of \eqref{est_B}. Finally, we get the last two estimates of \eqref{est_B} thanks to H\"older's inequality and triangular inequality, i.e.,
\[
\norm{\vjh\times \vBh}_{L^{2}L^1}
\aleq \norm{\vjh}_{L^{2}L^2}
 \norm{\vBh}_{L^{\infty}L^2} \aleq 1,
\]
and
\[
\begin{split}
& \norm{\vEh}_{L^{2}L^{3/2}}  =
\norm{\vjh - \vuh \times \vBh^\triangleleft}_{L^{2}L^{3/2}}
\aleq \norm{\vjh}_{L^{2}L^{3/2}} +
\norm{ \vuh \times \vBh^\triangleleft}_{L^{2}L^{3/2}} \\
&\aleq \norm{\vjh}_{L^{2}L^{2}} +
\norm{ \vuh }_{L^{2}L^6}
\norm{ \vBh}_{L^{\infty}L^{2}}
\lesssim    1 .
\end{split}
\]
\end{proof}

% To show the consistency of the numerical scheme we shall need further bounds on the numerical solution, which can be derived provided the adiabatic coefficient in \eqref{pressure} lies in the physically realistic range $\gamma\in(1,2).$
% \begin{Lemma}\label{ests2}
%  Let $(\vrh, \vuh)$ satisfy scheme \eqref{scheme}, $h\in(0,1)$ and  $\gamma \in(1,2)$. Then there hold
%  \begin{subequations}\label{est_rho_gamma12}
% \begin{align}
% %\label{est_srho_l2li}
% %\norm{\sqrt{\vrh}}_{L^2 L^\infty }  &\aleq h ^{ - \frac{\eps+2}{2\gamma}} , \\
% \label{est_rho_l2l2}
% \norm{\vrh}_{L^2 L^2 }  &\aleq h ^{ - \frac{\eps+2}{2\gamma}}, \\
% \label{est_m_l2l2}
%  \norm{\vrh \vuh} _{L^2L^2} &\aleq h^{- \frac{\eps+2}{2\gamma}}.
% \end{align}
% \end{subequations}
% \end{Lemma}

%%%%%%%%%%%%%%%%%%%%%%%%%%%%%%%%%%%% Consistency %%%%%%%
\subsubsection{Consistency}\label{sec_Consistency}
In this subsection, we aim to show the consistency of the numerical solutions of {\bf Scheme-I} \eqref{scheme}, which requires to replace the discrete test functions in \eqref{scheme} by the smooth test functions given in Definition~\ref{def_ca}. To this ends, we recall the following interpolation operators:
\begin{equation}\label{proj1}
\begin{aligned}
% &\PiQ \ :\ L^2(\Omega) \rightarrow Q_h,
% && \PiQ{\phi}|_K =\avK{\phi} \equiv \frac{1}{|K|} \intK{\phi},  \,\, \forall \; K \in \Oh,
% \\
& \PiV \ :\ W^{1,2}(\Omega) \rightarrow  \CR,
&& \intG \PiVv  =  \intG \vv  , \,\, \forall \; \sigma \in \mcE,
\\
& \PiN \ :\ W^{1,2}(\Omega) \rightarrow \ND,
&& \intG {\PiN \vv \times \bfn } =  \intG {\vv \times \bfn }  , \,\, \forall \; \sigma \in \mcE,
\\
& \PiR \ :\ W^{1,2}(\Omega) \rightarrow \RT,
&& \intG{ \PiR \vv \cdot \bfn}  =  \intG {\vv \cdot \bfn } , \,\, \forall \; \sigma \in \mcE.
\end{aligned}
\end{equation}
These interpolation operators satisfy the following interpolation error estimates~\cite{Brezzi,CR_elements}.
\begin{Lemma}\label{IE}
For any $\vv \in C^1(\Omega)$ $\vu \in C^2(\Omega)$, $p\in[1,\infty]$, the following hold
\[
\begin{aligned}
% \norm{\phi - \PiQ \phi}_{L^p} \aleq h \norm{\Grad  \vv}_{L^p},
% \quad
\norm{\vv - \PiVv }_{L^p}   \aleq h \norm{\vv}_{C^1},
\quad
\norm{\vv - \PiN \vv }_{L^p}   \aleq h \norm{\vv}_{C^1},
\quad
\norm{\vv - \PiR \vv }_{L^p}   \aleq h \norm{\vv}_{C^1},
\\
\norm{\Curl \vu -  \Curlh \PiN \vu }_{L^p}   \aleq h \norm{\vu}_{C^2},
\quad
\norm{\Curl \vu -  \Curlh \PiR \vu }_{L^p}   \aleq h \norm{\vu}_{C^2}.
\end{aligned}
\]
\end{Lemma}
% We also need a discrete variant of Poincar\'{e}'s inequality, see~\cite[Proposition 4.13]{Temam_NS}
% \begin{equation}\label{poin_ineq}
% \norm{\vv }_{L^6} \aleq \norm{\Grad  \vv}_{L^2},\  \forall\ \vv \in \CR.
% \end{equation}
\noindent
Now we are ready to prove the consistency of  {\bf Scheme-I}~\eqref{scheme} in the sense of Definition~\ref{def_ca}.
 \begin{Theorem}[Consistency of the solution of {\bf Scheme-I} \eqref{scheme}]\label{Tm2} \hspace{1em}\newline
 Let $(\vrh, \vuh,\vEh,\vBh)$ be a solution of {\bf Scheme-I} \eqref{scheme} on the time interval $[0,T]$ with $\TS\approx h$, $\gamma>4d /(1+3d)$ and the artificial diffusion parameter satisfy 
 \begin{equation}\label{eps}
\eps >0 \mbox{ if } \gamma \geq 2 \quad \mbox{ and }\quad   \eps \in(0, 2 \gamma-1 -d/3) \mbox{ if } \gamma \in(4d /(1+3d),2).
\end{equation} 
 Then \eqref{cP1}--\eqref{cP4} hold. 
% \begin{equation} \label{cP1}
% - \intO{ \vrh^0 \phi(0,\cdot) }  =
% \int_0^\tau \intO{ \left[ \vrh \partial_t \phi + \vrh \vuh \cdot \Grad \phi \right]} \dt  + \int_0^T
% e_{1,h} (t, \phi) \dt,
% \end{equation}
% for any $\phi \in C_c^2([0,T) \times \Ov{\Omega})$;
% \begin{equation} \label{cP2}
% \begin{split}
% - &\intO{ \vrh^0 \vuh^0 \vv(0,\cdot) }  =
% \int_0^T \intO{ \left[ \vrh \vuh \cdot \partial_t \vv + \vrh \vuh \otimes \vuh  : \Grad \vv  + p_h \Div \vv \right]} \dt,
% \\&
%  -   \int_0^T \intO{  \S( \Grad \vuh) : \Grad \vv}  \dt
% + \int_0^T \intO{ \vjh \times \vBh \cdot  \vv}\dt
%  + \int_0^T e_{2,h} (t, \vv) \dt
% \end{split}
% \end{equation}
% for any $\vv \in C^2_c([0,T] \times {\Omega}; R^d)$;
%
% \begin{equation} \label{cP3}
%  \int_0^T  \intO{\vjh \cdot \bfpsi -  \vBh \cdot(\Curl \bfpsi) }
%  = \int_0^T e_{3,h} (t, \bfpsi) \dt
% \end{equation}
% for any $\bfpsi \in C^2_c([0,T] \times {\Omega}; R^d)$;
%
% \begin{equation} \label{cP4}
%  \int_0^T  \intO{D_t \vBh \cdot \vC} + \intO{\vEh \cdot \Curl\vC}
%   = \int_0^T e_{4,h} (t, \bfpsi) \dt
% \end{equation}
%  for any $\vC \in C^2_c([0,T] \times {\Omega}; R^d)$;
%
% \[
% \| e_{j,h} (\cdot, \phi ) \|_{L^1(0,T)} \lesssim h^\beta   \norm{ \phi }_{C^2}  , \, j=1,2,3,4  \, \mbox{ for some }\ \beta > 0.
% \]
 \end{Theorem}
\begin{proof}
First, the consistency analysis of the Navier--Stokes part of {\bf Scheme-I} \eqref{scheme} has been shown in \cite{LiShe_MHD1}, meaning that \eqref{cP1} and \eqref{cP2} hold.  
%Thus we have \eqref{cP1} for free and there exists a $\beta>0$ such that
%\begin{equation}
%\begin{split}
%- &\intO{ \vrh^0 \vuh^0 \vv(0,\cdot) }  =
%\int_0^T \intO{ \left[ \vrh \vuh \cdot \partial_t \vv + \vrh \vuh \otimes \vuh  : \Grad \vv  + p_h \Div \vv \right]} \dt,
%\\&
% -   \int_0^T \intO{  \S( \Grad \vuh) : \Grad \vv}  \dt
%+ \int_0^T \intO{ \vjh \times \vBh \cdot  \PiV\vv}\dt
% + h^{\beta}.
%\end{split}
%\end{equation}
%Then we derive \eqref{cP2} by the following estimates
%\[
% \int_0^T \intO{ \vjh \times \vBh \cdot (\PiV\vv -\vv)}\dt
% \aleq \norm{\vjh}_{L^2L^2} \norm{\vBh}_{L^2L^2} h \norm{\vv}_{C^1}
%\aleq h.
%\]
%where we have used H\"older's inequality, the uniform bounds \eqref{est_u} and \eqref{est_B} as well as Lemma~\ref{IE}.
Thus we only need to prove \eqref{cP3} and \eqref{cP4}. 
In order to get \eqref{cP3}, we set $\vCh=\PiR \vC$ as the test function in \eqref{scheme_B} and analyze each term in the following. For the time derivative term we have
\[ \begin{aligned}
&\int_0^T  \intO{D_t \vBh \cdot \PiR \vC}
=\frac{1}{\TS}\int_0^T  \intO{ \vBh(t)-\vBh(t-\TS) \cdot \PiR \vC(t)}
\\&
=\frac{1}{\TS}\int_0^T  \intO{ \vBh(t)\cdot \PiR \vC(t)}
- \frac{1}{\TS}\int_{-\TS}^{T-\TS}  \intO{ \vBh(t)\cdot \PiR \vC(t+\TS)}
\\& =
-\int_0^T  \intO{ \vBh(t)\cdot D_t  \PiR \vC }
- \frac{1}{\TS}\int_{-\TS}^{0}  \intO{ \vBh(t)\cdot \PiR \vC(t+\TS)}
\\ & \qquad + \frac{1}{\TS}\int_{T-\TS}^{T}  \intO{ \vBh(t)\cdot \PiR \underbrace{\vC(t+\TS)}_{=0}}
\\& =
-\int_0^T  \intO{ \vBh(t)\cdot D_t  \PiR \vC }
-  \intO{ \vBh^0\cdot \PiR \vC(0)}
\\& =
-\int_0^T  \intO{ \vBh(t)\cdot \pdt \vC }
-  \intO{ \vBh^0\cdot \vC(0)}
+I_1+I_2
\end{aligned}
\]
where by H\"older's inequality and the estimate \eqref{est_B} we control
\[
I_1 = \int_0^T  \intO{ \vBh(t)\cdot (\pdt \vC - D_t  \PiR \vC) }\dt
\aleq  \norm{\vBh}_{L^2L^2} \TS \norm{\vC}_{C^2} \aleq h,
\]
and
\[
I_2 =   \intO{ \vBh(0)\cdot \big(\vC(0) - \PiR \vC(0)\big)}
\aleq  \norm{\vBh^0}_{L^1} \TS \norm{\vC}_{C^2} \aleq h.
\]

Next, using H\"older's inequality again with the estimate \eqref{est_B} and Lemma~\ref{IE} we derive
\[
\begin{aligned}
 \int_0^T  \intO{ \vEh \cdot \Curl (\PiR \vC -  \vC)}\dt
 \aleq  h\norm{\vC}_{C^2} \norm{\vEh}_{L^1L^1}
 \aleq  h\norm{\vC}_{C^2} \norm{\vEh}_{L^2L^{\frac32}}
 \aleq h.
\end{aligned}
\]
Collecting the above four formulae we obtain \eqref{cP3}.

Now we are left to show \eqref{cP4}. To proceed, we set $\bfpsi_h =\PiN \bfpsi$ as the test function in \eqref{scheme_E}. Then by H\"older's inequality, the uniform bounds \eqref{est_B} and Lemma~\ref{IE} we derive
\[ \int_0^T  \intO{\vjh \cdot (\bfpsi -\PiN \bfpsi )} \dt
 \aleq \norm{\vjh}_{L^2L^2} h\norm{\bfpsi}_{C^1}
 \aleq h,
\]
and
\[
 \int_0^T  \intO{ \vBh \cdot(\Curlh \PiN\bfpsi - \Curl \bfpsi ) }
 \aleq  \norm{\vBh}_{L^2L^2} h \norm{\bfpsi}_{C^2}
 \aleq h.
\]
Summing up the above two estimates, we finish the proof of \eqref{cP4} and the whole proof of consistency. 
\end{proof}

\subsubsection{Convergence}
Now we are ready to show the convergence of {\bf Scheme-I}.
\begin{Theorem}[Convergence of {\bf Scheme-I}]\label{Th_SA}
Let $(\vrh, \vuh,\vEh, \vBh)$ be a solution to {\bf Scheme-I}  with $\TS \approx h$, $\gamma > \frac{4d}{1+3d}$ and the artificial diffusion parameter satisfy \eqref{eps}.  Then it converges in the sense of Theorem~\ref{Th2}.
\end{Theorem}
\begin{proof}
First, it is obvious that the discrete operators $\Gradh$, $\Divh$ and $\Curlh$ are compatible with the corresponding continuous operator, see \cite[Section 13.4]{FeLMMiSh}. 
Next, combining Theorem~\ref{Tm1} (stability) and Theorem~\ref{Tm2} (consistency) we realize that the numerical solution of {\bf Scheme-I} is a consistent approximation of the MHD system in the sense of Definition~\ref{def_ca}. Applying Theorem~\ref{Th2} we derive the convergence for {\bf Scheme-I}.
\end{proof}

\subsection{Example II}
We have shown the convergence of a mixed FV--FE scheme~\eqref{scheme}.  Here, let us discuss another example by combining our magnetic--electric solver \eqref{scheme_B} and \eqref{scheme_E} with some suitable schemes for the Navier--Stokes equations.

\begin{tcolorbox}
\begin{Definition}[{\bf Scheme-II}]
Let $(\vr_0,\vu_0,\vB_0)$ be the initial data stated in \eqref{ini_c} and the discrete initial data be given by 
\begin{equation}\label{dini}
(\vrh^0,\vuh^0, \vBh^0)=(\PiQ \vr_0, \PiV \vu_0, \PiR \vB_0). 
\end{equation}
Let $\Omega_h$ be a uniform and regular mesh discretization of $\Omega$ consists of rectangles in 2D or cuboids in 3D. 
We say $(\vrh,\vuh, \vBh, \vEh) \in L_{\TS}(0, T;Q_h \times(Q_h)^d  \times \RT \times \ND)$ is a mixed finite volume--finite element approximation of the MHD system \eqref{pde}--\eqref{ini_c} if it satisfies 
 \eqref{scheme_B}, \eqref{scheme_E} and for all $K\in \grid$
\[
D_t \vr _K + \sum_{\sigma \in \facesK} \frac{|\sigma|}{|K|} \Fup (\vrh,\vuh) =0,
\]
\[
\begin{aligned}
D_t (\vrh \vuh)_K + \sum_{\sigma \in \facesK} \frac{|\sigma|}{|K|}
\left( \Fup(\vrh \vuh,\vuh)  
- \mu \frac{\jump{\vuh}}{d_\sigma}
+ \avg{p_h  -\nu  \Divh \vuh} \vc{n}\right)
 =- \frac{1}{|K|}\intK{\vjh \times \vBh^\triangleleft},
\end{aligned}
\]
where $d_\sigma$ is the distance of the centers of the elements $K$ and $L$ for all $\sigma=K|L$. Moreover, the artificial diffusion parameter satisfy 
 \begin{equation*}\label{eps}
\eps >0 \mbox{ if } \gamma \geq 2 \quad \mbox{ and }\quad   \eps \in(0, 2 \gamma-1 -d/3) \mbox{ if } \gamma \in(1,2).
\end{equation*} 
\end{Definition}
\end{tcolorbox}
Here, we recall \eqref{op_diff} and \eqref{num_flux} for the definition of the discrete operators  $\jump{\cdot}$, $\avg{\cdot}$ and the diffusive upwind numerical flux \eqref{num_flux}. Moreover, the discrete divergence operator for the piecewise constant velocity $\vuh \in (Q_h)^d$ is given by
\begin{equation*}
(\Div_h \vuh)_K =
\frac{1}{|K|}\sum_{\sigma\in \facesK}|\sigma| \avg{\vuh} \cdot \vc{n} \quad \forall \; K \in \grid.
\end{equation*}

Analogous to \cite{LiShe_MHD1}, the difference between {\bf Scheme-I} and {\bf Scheme-II} is the discretization of the Navier-Stokes part. Following the discussions in \cite[Remark 4.5]{LiShe_MHD1}, we present the following convergence result.
\begin{Proposition}[Convergence of {\bf Scheme-II}]\label{Th_SA}
 Let $(\vrh, \vuh,\vBh,\vEh)$ be a solution of {\bf Scheme-II} with $\TS \approx h$ and $\gamma > 1$.  Then it converges in the sense of Theorem~\ref{Th2}.
\end{Proposition}

\section{Conclusion}\label{sec_end}
In this paper we establish a general convergence theory for numerical approximations of the compressible MHD system \eqref{pde}.  
%The main tool we utilized are the recent concepts of dissipative weak solutions and consistent approximations. 
We have shown the convergence of consistent approximation towards a dissipative weak solution, gotether with the weak--strong uniqueness priciple, meaning a dissipative weak solution coincides with a classical solution of the same problem as long as they start from the same initial data. 
Interpreting the consistent approximation as the energy stability and consistency of suitable numerical solutions, we have built a generalized Lax equivalence theory:
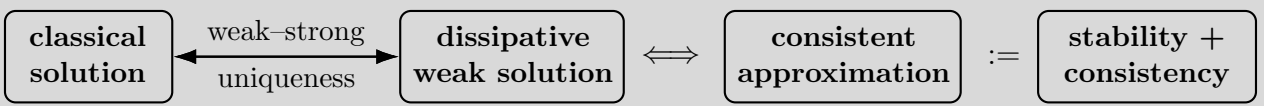
\begin{figure*}[!h]
\centering
\begin{tikzpicture}[scale=1,    box/.style = {draw, rounded corners,
                 minimum width=22mm, minimum height=5mm, align=center},
            > = {Latex[width=2mm,length=3mm]}]

            \draw[fill=Cgrey,Cgrey] (0,0.8) rectangle (0.985\textwidth,-0.9);
\node[draw,  thick] at (0.08\textwidth, 0)   (n1)  [box]{\hspace{-0.6cm} \bf \begin{tabular}{c}  classical \\ solution \end{tabular} \hspace{-0.75cm} };
\node[draw,  thick] at (0.4\textwidth, 0)   (n2)  [box]{\hspace{-0.3cm} \bf \begin{tabular}{c}  dissipative \\ weak solution \end{tabular}\hspace{-0.2cm}};
\node[draw,  thick] at (0.65\textwidth, 0)   (n3) [box]{\hspace{-0.3cm}  \bf \begin{tabular}{c}  consistent \\ approximation \end{tabular} \hspace{-0.5cm} };
\node[draw,  thick] at (0.88\textwidth, 0)   (n4) [box]{\bf \begin{tabular}{c}  stability +\\ consistency \end{tabular}};
\path[thick,<->]
            (n1)    edge  node[sloped, anchor=center, above] {weak--strong}          (n2)
            (n1)    edge  node[sloped, anchor=center, below] { uniqueness}          (n2);
\path node at (0.52\textwidth,0) {\large $\Longleftrightarrow$};
\path node at (0.772\textwidth,0) {\large $:=$};
\end{tikzpicture}
%\caption{S-limit.}\label{fig-Slimit}
\end{figure*}

Furthermore, in order to illustrate the application of this theory in the convergence analysis of numerical solutions, we proposed two mixed finite volume--finite element method. Our numerical schemes enjoy on the discrete level the energy stability, mass conservation, positivity of density, and the exact divergence free of magnetic field. 
By showing the solutions of the numerical schemes are consistent approximation, we conclude their convergence to a dissipative weak solution and the classical solution. This also indicates the global-in-time existence of dissipative weak solutions with general initial data for any adiabatic exponent $\gamma\in (1,\infty)$.

\clearpage\newpage

{\large{\centerline{\bf Acknowledgements}}}

The research of Y.~Li is supported by National Natural Science Foundation of China under grant No. 12001003. The research of B.~She is supported by Czech Science Foundation, grant No.~21-02411S. The institute of Mathematics of the Czech Academy of Sciences is supported by RVO:67985840.

\bibliographystyle{siamplain}

\end{document}